\documentclass[12pt,a4paper]{amsart}
\usepackage[latin1]{inputenc}   
\usepackage{dsfont}
\usepackage{amsmath}
\usepackage{amsfonts}
\usepackage{amssymb}
\usepackage{esint}
\usepackage{graphicx}
\usepackage{amsthm}
\usepackage{enumerate}
\usepackage{verbatim}
\usepackage{hyperref}
\usepackage[english]{babel}
\usepackage{bbm}
\usepackage{mathrsfs}
\usepackage{relsize}
\usepackage{textcomp}
\usepackage[leqno]{amsmath}
\usepackage{color}
\makeatletter
\newcommand{\leqnomode}{\tagsleft@true}
\newcommand{\reqnomode}{\tagsleft@false}
\makeatother

\def \R{\mathbb{R}}
\def \E{\mathbb{E}}
\def \P{\mathbb{P}}

\theoremstyle{plain} 
\newtheorem{thm}{Theorem}[section] 
 
\newtheorem{lem}[thm]{Lemma} 
\newtheorem{prop}[thm]{Proposition} 
\newtheorem{hp}[thm]{Hypotheses} 
\newtheorem{defn}[thm]{Definition}
 
\theoremstyle{definition} 

\numberwithin{equation}{section}

\newcommand{\miezz}{\frac{1}{2}}

\newcommand{\eps}{\varepsilon}

\newcommand{\into}{\ensuremath{\int_{\Omega}}}

\newcommand{\inti}{\ensuremath{\int_{0}^{t}\int_{\Omega}}}

\newcommand{\intif}{\ensuremath{\int_{0}^{T}\int_{\Omega}}}

\newcommand{\norm}[1]{\ensuremath{\left\Arrowvert #1 \right\Arrowvert}}
\newcommand{\norminf}[1]{\ensuremath{\left\Arrowvert #1 \right\Arrowvert_\infty}}
\newcommand{\indic}{\mathlarger{\mathbbm{1}}}

\newcommand{\dm}[1]{\ensuremath{\frac{\delta #1}{\delta m}}}
\newcommand{\dw}{\mathbf{d}_1}
\newcommand{\supt}{\sup\limits_{t\in[t_0,T]}}
\newcommand{\supo}{\sup\limits_{t\in[0,T]}}

\newcommand{\bdone}[1]{a(x)D#1\cdot\nu_{|\partial\Omega}=0}

\newcommand{\amu}{_{\frac{1+\alpha}{2},1+\alpha}}
\newcommand{\amd}{_{1+\frac{\alpha}{2},2+\alpha}}
\newcommand{\amv}{_{1,2+\alpha}}

\newcommand{\bo}[1]{\boldsymbol{#1}}
\newcommand{\meh}{m_{\bo{x}}^{N,i}}
\newcommand{\ui}{U_t^{N,i}}
\newcommand{\vi}{V_t^{N,i}}
\newcommand{\duij}{DU_t^{N,i,j}}
\newcommand{\dvij}{DV_t^{N,i,j}}
\newcommand{\duii}{DU_t^{N,i,i}}
\newcommand{\dvii}{DV_t^{N,i,i}}
\newcommand{\dujj}{DU_t^{N,j,j}}
\newcommand{\dvjj}{DV_t^{N,j,j}}
\newcommand{\uis}{U_s^{N,i}}
\newcommand{\vis}{V_s^{N,i}}
\newcommand{\duijs}{DU_s^{N,i,j}}
\newcommand{\dvijs}{DV_s^{N,i,j}}
\newcommand{\duiis}{DU_s^{N,i,i}}
\newcommand{\dviis}{DV_s^{N,i,i}}
\newcommand{\dujjs}{DU_s^{N,j,j}}
\newcommand{\dvjjs}{DV_s^{N,j,j}}

\newcommand{\be}{\begin{equation}}
	\newcommand{\ee}{\end{equation}}

\newcommand{\espot}{e^{-\alpha(\delta t+d(x)+d(y))}}
\newcommand{\espos}{e^{-\alpha(\delta s+d(x)+d(y))}}
\begin{document}
	
	\title{THE CONVERGENCE PROBLEM IN MEAN FIELD GAMES WITH NEUMANN BOUNDARY CONDITIONS}
	
	\author{Michele Ricciardi}\thanks{King Abdullah University of Sciences and Technologies (KAUST). Thuwal, Saudi Arabia.}

	\date{\today}

	\maketitle
	
	\begin{abstract}
		In this article we study the convergence of the Nash Equilibria in a $N$-player differential game towards the optimal strategies in the Mean Field Games, when the dynamic of the generic player includes a reflection process which guarantees the invariance of the state space $\Omega$. The well-posedness of the Master Equation allows us to use its solution $U$ in order to construct finite dimensional projections $u^N_i$, which will converge, in some suitable spaces, to the solution of the Nash system $v^N_i$.
	\end{abstract}

	\section{Introduction}
	
	This article is related to the convergence of Nash Equilibria in a $N$-players differential game, through the use of the so-called \emph{Master Equation}.
	
	The asymptotic behaviour of an $N$-players differential game is typically described by the Mean Field Games system, whose theory was introduced by J.-M- Lasry and P.-L.- Lions in 2006 (\cite{LL1SI, LL2SI, LL-japanSI, LL3SI}),  and in the same years by Caines, Huang and Malham\'{e}, see \cite{HCMSI}. Conversely, the study of the convergence problem is very often approached with an infinite dimensional equation equivalent to the Mean Field Games system: the Master Equation, whose definition was given by P.-L. Lions in his lectures at Coll\`{e}ge de France, \cite{prontoprontoprontoSI}.
	
	The complete novelty of our results is in the boundary conditions. The convergence problem is typically studied when the state space of the agents is the torus $\mathbb{T}^d$, or, especially in the probabilistic literature, the whole space $\R^d$. But in many applied models it is important to consider situations when the state space is a bounded domain $\Omega$, with conditions that force the trajectories of the players to remain in the domain. See, for example, the models proposed in \cite{golSI, gueantSI}.
	
	In our case the invariance of the domain is obtained with a \emph{reflections process at the boundary}. Namely, the dynamic of the player $i$, with $1\le i\le N$, is given by the following stochastic differential equation:
	
	\begin{equation}\label{dyn}
	\begin{cases}
	dX_t^i=b(X_t^i,\alpha_t^i)\,dt+\sqrt 2\sigma(X_t^i)dB_t^i-dk_t^i\,,\\
	X_{t_0}^i=x_0^i\,,
	\end{cases}
	\end{equation}
	where $\alpha_t^i$ is the control, chosen from a certain set $A$, $b:[0,T]\times\Omega\times A\to\R^d$ is the \emph{drift} function and $\sigma:\Omega\to\R^{d\times d}$ is the \emph{diffusion} matrix of the process.
	
	Moreover, $(B_t)^i$, $1\le i\le N$ are independent $d$-dimensional Brownian motions, $x_0^i\in\Omega$ and $k_t^i$ is a \emph{reflected process along the co-normal.} According to \cite{snitzmanSI}, this reflected process satisfies the following properties:
	$$
	k_t^i=\int_0^t a(X_s^i)\nu(X_s^i)\, d|k|_s^i\,,\qquad |k|_t^i=\int_0^t\indic_{X_s^i\in\partial\Omega}\, d|k|_s^i\,,
	$$
	where $a=\sigma\sigma^*$ and $\nu$ is the outward normal at $\partial\Omega$. This reflection along the co-normal forces the process to stay into $\Omega$ for all $t\ge0$.
	
	There is an extensive literature about stochastic differential equations with reflection, and existence results were already obtained, for example, in \cite{dieciSI, elshaaSI, 55SI, snitzmanSI, 63SI, 63bisSI, 64SI, 65SI}, so we will not discuss about it in this article.
	
	From now on, we will use the notation $\mathbf{v}$ to indicate a vector of $\R^{Nd}$ defined by $\mathbf{v}=(v^1,\dots,v^N)$, where $v^i$ is an already defined vector of $\R^d$.
	
	Assume that the cost for the player $i$ is given by the following functional:
	$$
	J^N_i(t_0,\bo{x}_0,\bo\alpha.)=\E\left[\int_{t_0}^T\left(L(X_s^i,\alpha_s^i)+F(X_s^i,m_{\bo{X_s}}^{N,i})\right)\,ds+G(X_T^i,m_{\bo{X_T}}^{N,i})\right]\,,
	$$
	
	where $L$ is the \emph{Lagrangian} cost for the control and $F$ and $G$ are the cost functions of the player $i$. Of course, there must be a symmetry structure for the system in order to have a convergence for $N\to+\infty$. Hence, we assume that the cost functions depend on the trajectory of the player $i$ and the empirical distribution of the other players, defined as follows:
	\begin{equation}\label{empirical}
	m_{\bo x}^{N,i}=\frac 1{N-1}\sum_{j\neq i}\delta_{x_j}\,,\quad\hbox{where $\delta_{x_j}$ is the Dirac function at $x_j$}\,.
	\end{equation}

	With these notations, a control $\bo{\alpha^*_\cdot}$ provides a Nash equilibrium if, for all controls $\bo\alpha.$ and for all $i$ we have
	$$
	J^N_i(t_0,\bo{x}_0,\bo{\alpha^*_\cdot})\le J^N_i\left(t_0,\bo{x}_0,\alpha_i,{(\alpha^*_j)}_{j\neq i}\right)\,,
	$$
	i.e., each player chooses his optimal strategy, if the other agents have chosen the control provided by the Nash equilibrium. The value function for the generic player $i$ corresponds to the cost functional evaluated at the optimal control:
	$$
	v^N_i(t_0,\bo x_0)=J^N_i(t_0,\bo x_0,\bo{\alpha^*_\cdot})\,.
	$$
	Using Ito's formula and the dynamic programming principle, one can prove the following: $\bo{\alpha^*}$ provides a Nash equilibrium if the related functions $v^N_i$ solve the so-called \emph{Nash system}:
	\begin{equation}\label{nash}
	\begin{cases}
	-\partial_t v^N_i(t,\bo x)-\mathlarger{\sum}\limits_j \mathrm{tr}(a(x_j)D^2_{x_jx_j}v^N_i(t,\bo x)) +H(x_i,D_{x_i}v^N_i(t,\bo x))\\
	\hspace{2.2cm}+\,\,\mathlarger{\sum}\limits_{j\neq i}H_p(x_j,D_{x_j}v^N_j(t,\bo x))\cdot D_{x_j}v^N_i(t,\bo x)=F(\bo x,m_{\bo x}^{N,i})\,,\\
	v^N_i(T,\bo x)=G(\bo x,m_{\bo x}^{N,i})\,,\\
	a(x_j)D_{x_j}v^{N,i}(t,\bo x)\cdot\nu(x_j)_{|x_j\in\partial\Omega}=0\,,
	\end{cases}
	\end{equation}
	for $(t,\bo x)\in[0,T]\times\R^{Nd}$. Here $a=\sigma\sigma^*$, $H_p(x,p)$ denotes $\frac{\partial H(x,p)}{\partial p}$, for $p\in\R^N$ and $H$ is the \emph{Hamiltonian} of the system, i.e. a slight variation of the Fenchel conjugate of the Lagrangian:
	$$
	H(x,p):=\sup\limits_{a\in A}\big(-b(x,\alpha)\cdot p- L(x,\alpha)\big)\,.
	$$
	
	Existence of solutions for this system is well known under some hypotheses of regularity and growth of the coefficients, see \cite{seiSI, lsuSI}.
	
	However, the structure of the $N$-players game becomes really intricate when $N\gg1$, and in that case we are naturally interested in an asymptotic behaviour of \eqref{nash} as $N\to+\infty$, in order to simplify the configuration of the Nash system.
	
	If we want to describe, at least heuristically, the structure of this limit problem when $N\to +\infty$, we find a differential game with infinitely many players, where the dynamic of a generic player is driven by a stochastic differential equation of this type:
	$$
	\begin{cases}
	dX_t=b(X_t,\alpha_t)\,dt+\sqrt 2\sigma(X_t)\,dB_t-dk_t\,,\\
	X_{t_0}=x_0\,,
	\end{cases}
	$$
	and each player chooses his own strategy in order to minimize
	$$
	J(t_0,x_0,\alpha_\cdot)=\E\left[\int_{t_0}^T\big(L(s,X_s,\alpha_s)+F(X_s,m(s))\big)\,ds+G(X_T,m(T))\right]\,,
	$$
	where $m(\cdot)$ is the density of the populations, obtained by the convergence of $m_{\bo x}^{N,i}$.
	
	In this case the Mean Field Games system takes the following form:
	\begin{equation}\label{mfg}
	\begin{cases}
	-u_t-\mathrm{tr}(a(x)D^2 u)+H(x,Du)=F(x,m(t))\,,\\
	m_t-\sum\limits_{i,j}\partial^2_{ij}(a_{ij}(x)m)-\mathrm{div}(mH_p(x,Du))=0\,,\\
	m(t_0)=0\,,\qquad u(T,x)=G(x,m(T))\,,\\
	\bdone{u}\,,\qquad \Big(\big(\sum\limits_j\partial_j(a_{ij}(x)m)\big)_i+H_p(x,Du)m\Big)\cdot \nu_{|\partial\Omega}=0\,,
	\end{cases}
	\end{equation}
	where a backward Hamilton-Jacobi-Bellman equation for the value function $u(t,x):=\inf\limits_\alpha J(t,x,\alpha)$ is coupled with a forward Fokker-Planck equation for the density of the population $m(\cdot)$.
	
	Once proved the results on the Mean Field Games problem, one naturally asks if this system can be a good approximation of the $N$-players system.
	
	In this context, two kind of results can be shown:
	\begin{itemize}
		\item [(i)] The optimal strategies in the Mean Field Games system provide approximated Nash equilibria (called $\eps$-Nash equilibria) in the $N$-players game.
		\item [(ii)] A Nash equilibrium in the $N$-player game converges, when $N\to+\infty$, towards an optimal strategy in the Mean Field Games.
	\end{itemize}

	The first question has been widely studied, using specific tools of the Mean Field theory. See, for instance, \cite{resultunoSI,resultdueSI,resulttreSI}, whereas many difficulties arise in the analysis of the second question, due to the lack of compactness properties of the problem.
	
	In order to overcome this problem, Lasry and Lions in \cite{prontoprontoprontoSI} introduced a new infinite dimensional equation, the so-called \emph{\textbf{Master Equation}}, which summarizes the whole Mean Field Games system in a unique equation and is clearly connected with the Nash system.
	
	The Master Equation is defined from its trajectories, which are solutions of the Mean Field Games system: if $(u,m)$ solves \eqref{mfg} with initial condition $m(t_0)=m_0$, we define
	\begin{equation}\label{defU}
	U:[0,T]\times\Omega\times\mathcal{P}(\Omega)\to\R\,,\qquad U(t_0,x,m_0)=u(t_0,x)\,,
	\end{equation}
	where $\Omega\subseteq\R^d$ is a bounded set, whose properties will be discussed later, and $\mathcal{P}(\Omega)$ is the set of Borel probability measures on $\Omega.$
	
	If we compute, at least formally, the equation satisfied by $U$, we obtain a non-local transport equation in the space of measures, called the \emph{Master Equation}:
	\begin{equation}\begin{split}
		\label{ME}
		\left\{
		\begin{array}{rl}
			&-\,\partial_t U(t,x,m)-\mathrm{tr}\left(a(x)D_x^2 U(t,x,m)\right)+H\left(x,D_x U(t,x,m)\right)\\&-\mathlarger{\into}\mathrm{tr}\left(a(y)D_y D_m U(t,x,m,y)\right)dm(y)\\&+\mathlarger{\into} D_m U(t,x,m,y)\cdot H_p(y,D_x U(t,y,m))dm(y)= F(x,m)\\&\mbox{in }(0,T)\times\Omega\times\mathcal{P}(\Omega)\,,\vspace{0.4cm}\\
			&U(T,x,m)=G(x,m)\hspace{1cm}\mbox{in }\Omega\times\mathcal{P}(\Omega)\,,\vspace{0.2cm}\\
			&a(x)D_x U(t,x,m)\cdot\nu(x)=0\hspace{1cm}\quad\,\mbox{for }(t,x,m)\in(0,T)\times\partial\Omega\times\mathcal{P}(\Omega)\,,\\
			&a(y)D_m U(t,x,m,y)\cdot\nu(y)=0\hspace{1cm}\mbox{for }(t,x,m,y)\in(0,T)\times\Omega\times\mathcal{P}(\Omega)\times\partial\Omega\,,
		\end{array}
		\right.
	\end{split}\end{equation}
	where $D_mU$ is a suitable derivative of $U$ with respect to the measure $m$ and where two Neumann boundary conditions for $U$ are satisfied, in the space and in the measure variable.
	
	The second boundary condition, introduced and studied in \cite{memedesimoSI}, is not a surprising fact: actually, the symmetric structure of the problem (i.e. the symmetry assumptions on $F$ and $G$) implies an exchangeability property for the functions $v^N_i$, which turn out to be invariant under permutations of ${(x_j)}_{j\neq i}$. Actually, one can prove that there exists a function $v^N:[0,T]\times\Omega\times\mathcal{P}(\Omega)$ such that
	\begin{equation}\label{vN}
	v^N_i(t,\bo{x})=v^N(t,x_i,\meh)\,.
	\end{equation}
	Hence, the boundary conditions of the Nash system \eqref{nash} are strictly related to the ones of the Master Equation \eqref{ME}. Loosely speaking:
	\begin{itemize}
		\item $a(x_i)D_{x_i}v^N_i(t,\bo x)\cdot\nu(x_i)_{|x_i\in\partial\Omega}=0$ corresponds to a Neumann boundary condition in the space variable, since $v^N$ depends on $x_i$ in the space variable;
		\item $a(x_j)D_{x_j}v^N_i(t,\bo x)\cdot\nu(x_j)_{|x_j\in\partial\Omega}=0$, for $j\neq i$, corresponds in the Master Equation to a Neumann boundary condition in $m$, since the dependence of $v^N$ with respect of $x_j$ is in the last variable, which is a measure variable.
	\end{itemize}

	Once defined the Master Equation, there are two main steps which must be handled:
	\begin{itemize}
		\item [(i)] Prove the well-posedness of the Master Equation: existence, uniqueness and regularity of solutions;
		\item [(ii)] Prove that any solution of the Nash system \eqref{nash} converge towards a solution of the Master Equation \eqref{ME}.
	\end{itemize}

	We point out here that these steps are studied in two different contexts: the first case is the so-called \emph{First order Master Equation}, when the control of the generic player has the form \eqref{dyn} and the Master Equation is \eqref{ME}, and the \emph{Second order Master Equation}, or \emph{Master Equation with common noise}. In this case, the dynamic \eqref{dyn} has also an additional Brownian term $dW_t$, not depending on $i$ (which justifies the adjective \emph{common}). This leads to a different and more difficult type of Master Equation, with some additional terms depending also on the second derivative $D_{mm}U$. It is relevant to say that Mean Field Games with common noise were already studied by Carmona, Delarue and Lacker in \cite{loackerSI}.
	
	The well-posedness of the Master Equation was studied in many papers. After the first definition given by Lions in \cite{prontoprontoprontoSI}, a first result of existence and uniqueness of solutions was proved, in the first order case, by Chassagneux, Crisan and Delarue in \cite{28SI}.
	
	The most important result in this context was certainly achieved by Cardaliaguet, Delarue, Lasry and Lions in \cite{cardSI}, who proved, in a periodic setting $\Omega=\mathbb{T}^d$, the well-posedness of the Master Equation in both first and second order cases.
	
	Other important results about the well-posedness of the Master Equation were given in \cite{gomezSI,bucchinSI,cardconvSI,fishSI,fifa21SI,tonaliSI,nuova16SI,koulibalySI}. Anyway, all these results are proved in the case $\Omega=\mathbb{T}^d$ or $\Omega=\R^d$, so they cannot be applied in our framework. A first important result about existence and uniqueness of solutions for the Master Equation with Neumann boundary condition was proved in \cite{memedesimoSI}, and the results contained in it will be the starting point of our work.
	
	As regards the convergence problem, the already mentioned work \cite{cardSI} for the well- posedness of the Master Equation contains also a convergence result, and the ideas used in it will be used also in this article. Of course, the presence of a boundary condition here requires more effort in order to obtain the same results.
	
	The symmetrical structure of the problem, in particular the expression of $v^N_i$ as in \eqref{vN},
	suggests to us to consider suitable finite dimensional projections of $U$, along the empirical
	distributions $\meh$. Hence, we define
	\begin{equation*}\
	u^N_i(t,\bo x):=U(t,x_i,\meh)
	\end{equation*}
	and the convergence problem holds in the sense that $|u^N_i-v^N_i|\to 0$ in some suitable norms. We will be more specific throughout the article.
	
	Many other results about the convergence problem are given in the literature. The convergence in the whole space, under weaker condition than in \cite{cardSI}, was given by Carmona and Delarue in \cite{resultunemmiezzSI}. In \cite{nuova14SI}, Cardaliaguet, Cirant and Porretta studied the convergence for the major-minor problem. Very important are the works of Delarue, Lacker and Ramanan, who used the Master Equation for the analysis of the large deviation problem and the central limit theorem, see \cite{nuova4SI,ramadanSI}. As regards finite state problems, some recents developments were studied by Bayraktar and Cohen in \cite{nuova1SI} and by Cecchin and Pelino in \cite{nuova11SI}.
	
	There are also convergence result obtained without using Master Equation. See, for example, the work by Lacker in \cite{loackerSI} and \cite{loacker2SI}. Other important papers about the convergence problem are \cite{sonocarmelaSI, cicciocaputoSI, checchinoSI, jedi, chicazzosieteSI, durrSI}.\\
	
	The paper is organized as follows.
	
	\begin{itemize}
		\item In Section $2$ we list the main notation and the hypotheses we need in the rest of the article;
		\item Section $3$ is devoted to the Master Equation. According to \cite{memedesimoSI}, we consider the solution of \eqref{ME} and we prove some further regularity estimates we will need in order to prove the convergence result;
		\item In Section $4$ we study the properties of the functions $u^N_i$ defined previously. In particular, we will give a representation formula for the derivatives of $u^N_i$ , depending on the derivatives of $U$; then we use these formulas in order to prove that $u^N_i$ solves ``almost" the Nash system, with an error of order $\frac1N\,$;
		\item In Section $5$ we define the following related process for $v^N_i$:
		$$
		\begin{cases}
		dY_t^i=-H_p(Y_t^i,D_{x_i}v^N_i(t,\bo Y_t))\,dt+\sqrt 2\sigma(Y_t^i)dB_t^i-dk_t^{i,Y}\,,\\
		Y_{t_0}^i=Z_i\,,
		\end{cases}
		$$
	where $\bo Z={(Z_i)}_i$ are i.i.d. random variables of fixed law $m_0$, and we prove that
	\begin{equation*}
	\begin{split}
	&|u^N_i(t_0,\bo Z)-v^N_i(t_0,\bo Z)|\le\frac CN\qquad\P-a.s.\,,\\
	\E&\left[\int_{t_0}^T|D_{x_i}v^N_i(t,\bo Y_t)-D_{x_i}u^N_i(t,\bo Y_t)|^2\,dt\right]\le\frac C{N^2}\,;
	\end{split}
	\end{equation*}
	With these asymptotic estimates, we are able to prove the two main convergence results: we will prove that
	$$
	\lim\limits_{N\to+\infty}\sup\limits_i|v^N_i(t_0,\bo x)-U(t_0,x_i,m_{\bo x}^N)|=0\,,
	$$
	with $m_{\bo x}^N:=\frac 1N\sum\limits_i\delta_{x_i}\,$; moreover, if we set
	$$
	w^N_i(t_0,x_i,m_0):=\int_{\Omega^{N-1}}v^N_i(t_0,\bo x)\prod\limits_{j\neq i}m_0(dx_j)\,,
	$$
	when, in $L^1(m_0)$ norm,
	$$
	\lim\limits_{N\to+\infty}w^N_i(t_0,\cdot,m_0)=U(t_0,\cdot,m_0)\,;
	$$
	\item Eventually, in Section $6$ we prove a last result concerning the \emph{convergence of the trajectories:} if we consider the process
	$$
	\begin{cases}
	dX_t^i=-H_p(X_t^i,D_{x_i}u^N_i(t,\bo X_t))\,dt+\sqrt 2\sigma(X_t^i)dB_t^i-dk_t^{i,X}\,,\\
	X_{t_0}^i=Z_i\,,
	\end{cases}
	$$
	then
	$$
	\E\left[\supt|X_t^i-Y_t^i|^2\right]\le\frac C{N^2}\,.
	$$
	\end{itemize}
	
	\section{Notation and Assumptions}
	
	Let $\Omega\subset\R^d$ be the closure of an open bounded set, with boundary $\mathcal{C}^{2+\alpha}$. Called $T$ the final time of the process, we define $Q_T:=[0,T]\times\Omega$.
	
	As already said, $\mathcal{P}(\Omega)$ denotes the set of Borel probability measures in $\Omega$. We define the \emph{push-forward} measure in this way: for $\Upsilon\subset\R^d$, $r:\Omega\to\Upsilon$ a Borel map and $\mu\in\mathcal{P}(\Omega)$, the push-forward measure $r\sharp\mu\in\mathcal{P}(\Upsilon)$ is defined by $r\sharp\mu (A)=\mu(r^{-1}(A))$, for $A\subseteq\Upsilon\,.$
	
	We briefly recall the definitions of the spaces of functions involved in this article, already stated in \cite{memedesimoSI}.
	
	For $\alpha\in(0,1)$ and $n\ge0$, the space $\mathcal{C}^{n+\alpha}(\Omega)$ of functions $n$-times differentiable with $\alpha$-H\"{o}lder continuous derivatives is endowed with the following norm:
	\begin{align*}
	\norm{\phi}_{n+\alpha}:=\sum\limits_{|\ell|\le n}\norminf{D^l\phi}+\sum\limits_{|\ell|= n}\sup\limits_{x\neq y}\frac{|D^\ell\phi(x)-D^\ell\phi(y)|}{|x-y|^\alpha}\,.
	\end{align*}
	The subspace of $\mathcal{C}^{n+\alpha}(\Omega)$ consisting of functions $\phi$ such that $\bdone{\phi}$ is denoted by $\mathcal{C}^{n+\alpha,N}(\Omega)$.
	
	In the same way we can define the space of functions $\mathcal{C}^{\frac{n+\alpha}{2},n+\alpha}([0,T]\times\Omega)$, consisting of continuous functions $\phi$ with $\left(\frac\alpha 2,\alpha\right)$-H\"{o}lder derivatives $D_t^rD_x^s\phi$, with $2r+s\le n$. When there is no possibility of confusion, we will write simply $\mathcal{C}^{n+\alpha}$, $\mathcal{C}^{\frac{n+\alpha}{2},n+\alpha}$.
	
	In a similar way we define $\mathcal{C}^{0,\alpha}$, $\mathcal{C}^{\alpha,0}$, $\mathcal{C}^{1,2+\alpha}$. For a precise definition of the spaces and the norms endowed, we refer to \cite{lsuSI,memedesimoSI}.
	
	The dual spaces of $\mathcal{C}^{n+\alpha}$ and $\mathcal{C}^{n+\alpha,N}$ will be denoted by $\mathcal{C}^{-(n+\alpha)}$ and $\mathcal{C}^{-(n+\alpha),N}$, endowed with the classical duality norm:
	$$
	\norm{\rho}_{-(n+\alpha)}=\sup\limits_{\norm{\phi}_{n+\alpha}\le 1}\langle\rho,\phi\rangle\,,\qquad \norm{\rho}_{-(n+\alpha),N}=\sup\limits_{\substack{\norm{\phi}_{n+\alpha}\le 1\\aD\phi\cdot\nu_{|\partial\Omega}=0}}\langle \rho,\phi \rangle\,.
	$$
	
	In order to give a notion of continuity and differentiability with respect to the measure $m$, we need to give two important definitions: the \emph{Wasserstein distance} and the \emph{intrinsic derivative}.
	\begin{defn}
	We define the Wasserstein distance $\dw$ in $\mathcal{P}(\Omega)$ in the following way: for $m_1\,,m_2\in\mathcal{P}(\Omega)$
	\begin{equation}\label{def:wass}
	\dw(m_1,m_2):=\sup\limits_{Lip(\phi)\le 1}\into \phi(x)d(m_1-m_2)(x)\,.
	\end{equation}
	\end{defn}
	
	Now we define two suitable derivations of $U$ with respect to the measure $m$. The second one appears in the formulation of the Master Equation, giving sense to \eqref{ME}.
	
	\begin{defn}
	Let $U:\mathcal{P}(\Omega)\to\R$. We say that $U$ is of class $\mathcal{C}^1$ if there exists a continuous map $K:\mathcal{P}(\Omega)\times\Omega\to\R$ such that, for all $m_1$, $m_2\in\mathcal{P}(\Omega)$ we have
	\begin{equation}\label{deu}
	\lim\limits_{t\to0}\frac{U(m_1+t(m_2-m_1))-U(m_1)}t=\into K(m_1,x)\,d(m_2-m_1)(x)\,.
	\end{equation}
	We call $\dm{U}(m,x):=K(m,x)$. Then, if $U$ is of class $\mathcal{C}^1$ and $\dm{U}$ is $\mathcal{C}^1$ with respect to the space variable $x$, we can define the intrinsic derivative $D_mU:\mathcal{P}(\Omega)\times\Omega\to\R^d$ as
	$$
	D_mU(m,x):=D_x\dm{U}(m,x)\,.
	$$ 
	\end{defn}
	
	We observe that $\dm{U}$ is defined in \eqref{deu} up to additive constants. Therefore we adopt the following normalization convention
	$$
	\into\dm{U}(m,x)dm(x)=0\qquad\forall\, m\in\mathcal{P}(\Omega)\,.
	$$
	From \eqref{deu} we deduce a sort of first-order Taylor expansion in $m$: if $m_1,\,m_2\in\mathcal{P}(\Omega)$ then
	$$
	U(m_2)-U(m_1)=\int_0^1\into\dm{U}(m_1+s(m_2-m_1),x)\,d(m_2-m_1)(x)\,.
	$$
	We conclude this section by stating the main assumptions we will need in the paper.
	
	\begin{hp}\label{ipotesi}
	Suppose that, for some $\alpha\in(0,1)$ and $C>0$,
	\begin{itemize}
		\item [a.] $\norm{a(\cdot)}_{1+\alpha}<\infty$ and, for some $\mu>\lambda>0$ it holds $\lambda I_{d\times d}\le a(x)\le \mu I_{d\times d}\,\quad\forall\,x\in\R^d\,.$
		\item [b.]$H:\Omega\times\R^d\to\R$ is a smooth function, Lipschitz in the last variable and s.t.
		$$
		0< H_{pp}(x,p)\le C I_{d\times d}\qquad\mbox{for a certain }C>0\,;
		$$
		\item [c.] $F$ and $G:\Omega\times\mathcal{P}(\Omega)\to\R$ are smooth functions, $\mathcal{C}^1$ in the measure variable and satisfying
		$$
		\into \left(F(x,m')-F(x,m)\right) d(m'-m)(x)\ge0\,,\qquad (\mbox{resp. with }G)\,,
		$$
		Moreover, $F,\,G,\,\dm{F}$ and $\dm{G}$ satisfy the following estimates:
		$$
		\sup\limits_{m\in\mathcal{P}(\Omega)}\left(\norm{F(\cdot,m)}_{\alpha}+\norm{G(\cdot,m)}_{2+\alpha}+\norm{\frac{\delta F}{\delta m}(\cdot,m,\cdot)}_{\alpha,2+\alpha}\hspace{-0.3cm}+\norm{\frac{\delta G}{\delta m}(\cdot,m,\cdot)}_{2+\alpha,2+\alpha}\right)\le C\,.
		$$
		\begin{align*}
		&\mathrm{Lip}\left(\dm{F}\right):=\sup\limits_{m_1\neq m_2}\left(\dw(m_1,m_2)^{-1}\norm{\dm{F}(\cdot,m_1,\cdot)-\dm{F}(\cdot,m_2,\cdot)}_{\alpha,1+\alpha}\right)<+\infty\,,\\
		&\mathrm{Lip}\left(\dm{G}\right):=\sup\limits_{m_1\neq m_2}\left(\dw(m_1,m_2)^{-1}\norm{\dm{G}(\cdot,m_1,\cdot)-\dm{G}(\cdot,m_2,\cdot)}_{2+\alpha,2+\alpha}\right)<+\infty\,;
		\end{align*}
		\item [d.] The following Neumann boundary conditions are satisfied:
		\begin{align*}
			&\left\langle a(y)D_y\dm{F}(x,m,y), \nu(y)\right\rangle_{|\partial\Omega}=0\,,\qquad \left\langle a(y)D_y\dm{G}(x,m,y),\nu(y)\right\rangle_{|\partial\Omega}=0\,,\\
			&\langle a(x)D_xG(x,m), \nu(x)\rangle_{|\partial\Omega}=0\,,
		\end{align*}
		for all $m\in\mathcal{P}(\Omega)$.
		\end{itemize}
	\end{hp}

	We stress the fact that the first boundary condition appearing in hypothesis $d.$ is not a classical compatibility condition, but it will be crucial in order to prove to prove the Neumann boundary condition of $D_mU$ in \eqref{ME}, see \emph{Corollary 5.13} of \cite{memedesimoSI} for further details.
	
	With these hypotheses we will be able to prove the main convergence result of the paper, which is the following one:
	\begin{thm}\label{thm:Teorema}
	Suppose hypotheses \ref{ipotesi} hold true. Then, if we define
	\begin{equation}\label{eq:defmN}
	m_{\bo x}^N:=\frac 1N\sum\limits_{i=1}^N\delta_{x_i}\,,
	\end{equation}
	we have
	\begin{equation}\label{eq:risult1}
	\sup\limits_{\substack{t\in[0,T]\\i=1,\dots,N}}\big|v^N_i(t,\bo{x})-U(t,x_i,m_{\bo x}^N)\big|\le\frac CN\,.
	\end{equation}
	Moreover, if we set
	\begin{equation}\label{eq:defwN}
	w^N_i(t,x_i,m):=\int_{\Omega^{N-1}}v^N_i(t,\bo{x})\prod\limits_{j\neq i}m_0(dx_j)\,,
	\end{equation}
	then
	\begin{equation}\label{eq:risult2}
	\norm{w^N_i(t,\cdot,m)-U(t,\cdot,m)}_{L^1(m)}\le C\omega_N\,,
	\end{equation}
	where
	\begin{equation}\label{eq:defomega}
	\omega_N=\left\{
	\begin{array}{ll}
	CN^{-\frac 1d} & \mbox{if }d\ge 3\,,\\
	CN^{-\frac 12}\log N & \mbox{if }d=2\,,\\	
	CN^{-\frac 12} & \mbox{if }d=1\,.\\
	\end{array}
	\right.
	\end{equation}
	\end{thm}
	Although the proof of this Theorem is the same as \emph{Theorem 2.13} of \cite{cardSI}, all the regularity results about the solution of the Master Equation and its trajectories $u^N_i$ needs to be readapted in the case of Neumann boundary conditions.
	
	\section{The Master Equation and further estimates}
	
	In this section we recollect some basic results about the Mean Field Games system with Neumann conditions and the well-posedness of the Master Equation. Then we will use these results in order to improve the regularity of $U$ and to prove some technical estimates we will need in the rest of the paper.\\
	
	The first result guarantees existence, uniqueness and regularity of solutions for the Mean Field Games system \eqref{mfg}.
	
	\begin{thm}
	Suppose hypotheses \ref{ipotesi} hold. Then the system \eqref{mfg} has a unique classical solution $(u.m)\in\mathcal{C}^{1+\frac\alpha 2,2+\alpha}\times\mathcal{C}([0,T];\mathcal{P}(\Omega))$, and this solution satisfies, for a certain $C>0$,
	\begin{equation}\label{stimemfg}
	\sup\limits_{t_1\neq t_2}\frac{\dw(m(t_1),m(t_2))}{|t_1-t_2|^\miezz}+\norm{u}\amd\le C\,.
	\end{equation}
	Furthermore, if $(u_1,m_1)$ and $(u_2,m_2)$ are two solutions of \eqref{mfg}, with $m_1(t_0)=m_{01}$, $m_2(t_0)=m_{02}$, then for a certain $C>0$, $p>1$,
	\begin{equation}\label{eq:stimeLip}
	\norm{u_1-u_2}\amv+\supt\dw(m_1(t),m_2(t))+\norm{m_1-m_2}_{L^p(Q_T)}\le C\dw(m_{01},m_{02})\,,	
	\end{equation}
where $C$ does not depend on $t_0,m_{01},m_{02}\,.$
	\begin{proof}
	See \emph{Proposition 3.3.}, \emph{Proposition 4.1.} and \emph{Corollary 5.6.} of \cite{memedesimoSI}.
	\end{proof}
	\end{thm}

	We define $U$ as in \eqref{defU}. We recall that the crucial step in order to prove the well-posedness of the Master Equation is the $\mathcal{C}^1$ character of $U$ with respect to $m$. Actually, the regularity w.r.t. the measure is strictly related to the study of the following MFG linearized system:
	
	\begin{equation}\label{linear}
		\begin{cases}
			-z_t-\mathrm{tr}(a(x)D^2z)+H_p(x,Du)Dz=\mathlarger{\dm{F}}(x,m(t))(\rho(t))+h(t,x)\,,\\
			\rho_t-\sum\limits_{i,j}\partial^2_{ij}(a_{ij}(x)\rho)-\mathrm{div}(\rho(H_p(x,Du)+\tilde{b}))-\mathrm{div}(m H_{pp}(x,Du) Dz+c)=0\,,\\
			z(T,x)=\mathlarger{\dm{G}}(x,m(T))(\rho(T))+z_T(x)\,,\qquad\rho(t_0)=\rho_0\,,\\
			a(x)Dz\cdot\nu_{|\partial\Omega}=0\,,\quad\Big(\big(\sum\limits_j\partial_j(a_{ij}(x)\rho)\big)_i+\rho H_p(x,Du)+mH_{pp}(x,Du) Dz+c\Big)\cdot\nu_{|\partial\Omega}=0\,,
		\end{cases}
	\end{equation}
	
	\begin{thm}
	Suppose hypotheses \ref{ipotesi} hold, and let $z_T\in\mathcal{C}^{2+\alpha,N}$, $\rho_0\in\mathcal{C}^{-(1+\alpha)}$, $h\in\mathcal{C}^{0,\alpha}([t_0,T]\times\Omega)$, $c\in L^1([t_0,T]\times\Omega)$. Then there exists a unique solution $(z,\rho)\in\mathcal{C}^{1,2+\alpha}\times\,\left(\mathcal{C}([0,T];\mathcal{C}^{-(1+\alpha),N}(\Omega))\cap L^1(Q_T)\right)$ of system \eqref{linear}, which satisfies, for a certain $p>1$,
	\begin{equation}\label{eq:stimelin}
			\norm{z}\amv+\supt\norm{\rho(t)}_{-(1+\alpha),N}+\norm{\rho}_{L^p}\le C \left(\norm{z_T}_{2+\alpha}+\norm{\rho_0}_{-(1+\alpha)}+\norm{h}_{0,\alpha}+\norm{c}_{L^1}\right)\,.
	\end{equation}
	Moreover, $U$ is $\mathcal{C}^1$ w.r.t. $m$ and, called $(v,\mu)$ the solution of \eqref{linear} with $h=c=z_T=0$ and $\rho_0=\mu_0\in\mathcal{C}^{-(1+\alpha)}$, we have
	\begin{equation}\label{eq:reprformula}
	v(t_0,x)=\left\langle \dm{U}(t_0,x,m_0,\cdot),\mu_0\right\rangle\,,
	\end{equation}
	where $\langle\cdot,\cdot\rangle$ denotes the duality between $\mathcal{C}^{-(1+\alpha),N}$ and $\mathcal{C}^{1+\alpha,N}$. Finally, $\dm{U}$ satisfies
	\begin{equation}\label{eq:regdmU}
	\sup\limits_{\substack{t\in[0,T]\\m\in\mathcal{P}(\Omega)}}\norm{\dm{U}(t,\cdot,m,\cdot)}_{2+\alpha,2+\alpha}\le C\,.
	\end{equation}
	\begin{proof}
	See \emph{Proposition 5.8.}, \emph{Proposition 5.9.},  \emph{Theorem 5.10.}, \emph{Corollary 5.12.} of \cite{memedesimoSI}.
	\end{proof}
	\end{thm}

	Throughout the paper, we will use the notation $(z,\rho)$ to indicate a general solution of \eqref{linear}, and the notation $(v,\mu)$ to indicate the solution of \eqref{linear} with $z_T=c=h=0$ and $\rho_0=\mu_0\,$. 

	The $\mathcal{C}^1$ character of the function $U$ in the measure variable allows us to prove existence and uniqueness of solutions for the Master Equation \eqref{ME} (\emph{Theorem 2.5.} of \cite{memedesimoSI}). The next result allows us to improve the regularity of the function $U$. This improvement will be essential to prove some regularity estimates for the projections ${(u^N_i)}_i\,$.
	
	\begin{thm}
	Suppose hypotheses \ref{ipotesi} are satisfied. Then the derivative of the Master Equation $\dm{U}$ is Lipschitz continuous with respect to the measure variable:
	\begin{equation}\label{eq:LipdmU}
	\supo\sup\limits_{m_1\neq m_2}\big(\dw(m_1,m_2)\big)^{-1}\norm{\dm{U}(t,\cdot,m_1,\cdot)-\dm{U}(t,\cdot,m_2,\cdot)}_{2+\alpha,1+\alpha}\le C\,.
	\end{equation}

	\begin{proof}
	We consider, for $i=1,2$, the solution $(v_i,\mu_i)$ of the linearized system \eqref{linear} related to $(u_i,m_i)\,.$
	
	To avoid too heavy notations, we take $t_0=0$ and we define
	\begin{align*}
	\begin{array}{ll}
		H'_i(t,x):=H_p(x,Du_i(t,x))\,, & H''_i(t,x)=H_{pp}(x,Du_i(t,x))\,,\\
		F'(x,m,\mu)=\displaystyle\into \dm{F}(x,m,y)\,\mu(dy)\,, & G'(x,m,\mu)=\displaystyle\into \dm{G}(x,m,y)\,\mu(dy)\,. 
	\end{array}
	\end{align*}

	Then the couple $(z,\rho):=(v_1-v_2,\mu_1-\mu_2)$ satisfies the following linear system:
	$$
	\begin{cases}
	-z_t-\mathrm{tr}(a(x)D^2z)+H'_1\cdot Dz=F'(x,m_1(t),\rho(t))+h\,,\\
	\rho_t-\sum\limits_{i,j}\partial^2_{ij}(a_{ij}(x)\rho)-\mathrm{div}(\rho H'_1)-\mathrm{div}(m_1H''_1Dz+c)=0\,,\\
	z(T,x)=G'(x,m_1(T),\rho(T))+z_T\,,\qquad \rho(t_0)=0\,,\\
	a(x)Dz\cdot\nu_{|\partial\Omega}=0\,,\qquad \Big(\big(\sum\limits_j\partial_j(a_{ij}(x)\rho)\big)_i+\rho H'_1+mH''_1Dz+c\Big)\cdot \nu_{|\partial\Omega}=0\,,
	\end{cases}
	$$
	where
	\begin{align*}
	&h(t,x)=h_1(t,x)+h_2(t,x)\,,\\
	&h_1(t,x)=F'(x,m_1(t),\mu_2(t))-F'(x,m_2(t),\mu_2(t))\,,\\
	&h_2(t,x)=\big(H'_1(t,x)-H'_2(t,x)\big)\cdot Dv_2(t,x)\,,\\
	&c(t,x)=\mu_2(t)\big(H'_1-H'_2)(t,x)+\big[(m_1H''_1-m_2H''_2)\big](t,x)\,,\\
	&z_T(x)=G'(x,m_1(T),\mu_2(T))-G'(x,m_2(T),\mu_2(T))\,.
	\end{align*}

	Applying \eqref{eq:stimelin} we obtain this estimate on $z$:
	$$
	\norm{z}\amv\le C\left(\norm{z_T}_{2+\alpha}+\norm{h}_{0,\alpha}+\norm{c}_{L^1}\right)\,.
	$$
	
	Now we estimate the terms in the right-hand side.
	
	The term with $z_T$, thanks to \eqref{eq:stimelin} and the hypothesis \emph{c.} of \ref{ipotesi}, is immediately estimated:
	$$
	\norm{z_T}_{2+\alpha}\le\norm{\dm{G}(\cdot,m_1(T),\cdot)-\dm{G}(\cdot,m_2(T),\cdot)}_{2+\alpha,1+\alpha}\norm{\mu_2(T)}_{-(1+\alpha),N}\le C\dw(m_{01},m_{02})\norm{\mu_0}_{-(1+\alpha)}\,.
	$$
	As regards the space estimate for $h$, we have
	$$
	\norm{h(t,\cdot)}_\alpha\le\norm{F'(\cdot,m_1(t),\mu_2(t))-F'(\cdot,m_2(t),\mu_2(t))}_\alpha+\norm{(H'_1-H'_2)(t,\cdot)Dv_2(t,\cdot)}_\alpha\,.
	$$
	The first term is bounded as $z_T:$
	$$
	\norm{F'(\cdot,m_1(t),\mu_2(t))-F'(\cdot,m_2(t),\mu_2(t))}_\alpha\le C\dw(m_{01},m_{02})\norm{\mu_0}_{-(1+\alpha)}\,.
	$$
	The second term, using \eqref{eq:stimeLip} and \eqref{eq:stimelin}, can be estimated in this way:
	$$
	\norm{(H'_1-H'_2)(t,\cdot)Dv_2(t,\cdot)}_\alpha\le C\norm{(u_1-u_2)(t)}_{1+\alpha}\norm{v_2(t)}_{1+\alpha}\le C\dw(m_{01},m_{02})\norm{\mu_0}_{-(1+\alpha)}\,.
	$$
	In summary,
	$$
	\norm{h}_{0,\alpha}=\supo\norm{h(t,\cdot)}_\alpha\le C\dw(m_{01},m_{02})\norm{\mu_0}_{-(1+\alpha)}\,.
	$$
	Finally, we estimate $\norm{c}_{L^1}$. We have
	\begin{align*}
	\norm{c}_{L^1} &=\intif(H'_1-H'_2)(t,x)\,\mu_2(t,dx)\,dt+\intif H''_1(t,x)Dv_2(t,x)(m_1(t)-m_2(t))(dx)\,dt\\
	&+\intif(H'_1-H'_2)(t,x)\,Dv_2(t,x)\,m_2(t,dx)\,dt\le C\norm{u_1-u_2}\amu\norm{\mu_2}_{L^1}\\
	&+C\norm{u_1}\amu\norm{v_2}\amu\norm{m_1-m_2}_{L^1}+C\norm{u_1-u_2}\amu\norm{v_2}\amu\,.
	\end{align*}
	The first term in the right-hand side, thanks to \eqref{eq:stimeLip} and \eqref{eq:stimelin}, is bounded by
	$$
	C\norm{u_1-u_2}\amu\norm{\mu_2}_{L^1}\le C\dw(m_{01},m_{02})\norm{\mu_0}_{-(1+\alpha)}\,.
	$$
	The second and the third term are estimated in the same way, using \eqref{stimemfg} and again \eqref{eq:stimeLip} and \eqref{eq:stimelin}. Then
	$$
	\norm{c}_{L^1}\le C\dw(m_{01},m_{02})\norm{\mu_0}_{-(1+\alpha)}\,.
	$$
	Putting together all these estimates, we finally obtain:
	$$
	\norm{z}\amv\le C\dw(m_{01},m_{02})\norm{\mu_0}_{-(1+\alpha)}\,.
	$$
	Since
	$$
	z(t_0,x)=\into\left(\dm{U}(t_0,x,m_1,y)-\dm{U(t_0,x,m_2,y)}\right)\,\mu_0(dy)\,,
	$$
	we have proved \eqref{eq:LipdmU}\,.
	\end{proof}
	\end{thm}

	This theorem is a fundamental step in order to prove this technical lemma.
	
	\begin{lem}\label{lem:lemmanodim}
	Suppose hypotheses of the previous theorem are satisfied. Then, if $m\in\mathcal{P}(\Omega)$ and $\phi\in L^2(m,\R^d)$ is a bounded vector field such that \emph{Im}$(id+\phi)\subseteq\Omega$, we have:
	\begin{equation}\label{eq:primolemma}
	\norm{U(t,\cdot,(id+\phi)\sharp m)-U(t,\cdot,m)-\into D_mU(t,\cdot, m,y)\cdot\phi(y)\,dm(y)}_{1+\alpha}\le C\norm{\phi}^2_{L^2(m)}\,.
	\end{equation}
	\end{lem}

	The proof is a trivial readaptation of \emph{Proposition 7.3.} and \emph{Proposition 7.4.} of \cite{cardSI}, so we skip it.
	
	\section{The projections $u^N_i$ and their properties}
	
	As already said in the introduction, in order to prove the convergence of $v^N_i$ towards $U$, the main idea si to work with suitable \emph{finite dimensional projections} of $U$, proving that they are nearly solutions to the Nash system.
	
	So, for $N\ge 2$ and $1\le i\le N$, we define the following functions $u^N_i$:
	\begin{equation}\label{uN}
	u^N_i(t,x):=U(t,x_i,\meh)\,,
	\end{equation}
	where $\meh$ is the empirical distribution of the players $j\neq i$, defined in \eqref{empirical}\,.
	
	Thanks to the regularity of $U$, we already know that
	\begin{equation}\label{eq:reg_uN_1}
	u^N_i\in\mathcal{C}^{1+\frac\alpha 2,2+\alpha}\qquad\hbox{with respect to the couple $(t,x_i)$}\,.
	\end{equation}
	Using Lemma \ref{lem:lemmanodim} we are able to prove a regularity result for $u^N_i$ with respect to the other variables $(x_j)_{j\neq i}\,$.
	
	\begin{prop}\label{prop:reg_uN_2}
	For all $j\neq i$, the following formulas for the derivatives of $u^N_i$ hold true:
	\begin{align}
	&\hspace{0.6cm}D_{x_j}u^N_i(t,\bo{x}) =\frac 1{N-1}D_mU(t,x_i,\meh,x_j)\,,\label{eq:reg_uN_2}\\
	&\hspace{0.2cm}D^2_{x_i, x_j}u^N_i(t,\bo{x}) =\frac 1{N-1}D_xD_mU(t,x_i,\meh,x_j)\,,\label{eq:reg_uN_3}\\
	&\left|D^2_{x_j,x_j}u^N_i(t,\bo{x})-\frac 1{N-1}D_yD_mU(t,x_i,\meh,x_j)\right| \le\frac C{N^2}\,. \label{eq:reg_uN_4}
	\end{align}
	\begin{proof}
	Thanks to the regularity of $D_mU$, the second equality is an obvious consequence of the first one. So, we restrict ourselves to the proof of the first and the third formula.
	
	We consider $\bo{x}=(x_1,\dots,x_N)\in\Omega^N$ with $x_j\neq x_k$ when $j\neq k$ and $\eps<\min\limits_{j\neq k}|x_j-x_k|\,$.
	
	We fix a vector $\bo{v}=(v_1,\dots,v_N)$ with $v_i=0$ and $\bo{x}+\bo{v}\in\Omega^N$, and we consider, for $\eps$ small enough, a smooth and bounded vector field such that
	$$
	\phi(x)=v_j\qquad\forall\,x\in B_{\frac\eps 3}(x_j)\,,\qquad\qquad x+\phi(x)\in\Omega\qquad\forall\, x\in\Omega\,.
	$$
	We note that
	$$
	u^N_i(t,\bo{x}+\bo{v})=U(t,x_i,(id+\phi)\sharp\meh)\,,\qquad u^N_i(t,\bo{x})=U(t,x_i,\meh)\,.
	$$
	
	Then, \eqref{eq:primolemma} implies that
	$$
	u^N_i(t,\bo{x}+\bo{v})=u^N_i(t,\bo{x})+\into D_m U(t,x_i,\meh,y)\cdot\phi(y)\,d\meh(y)+o\left(\norm{v}_{L^2\left(\meh\right)}\right)\,.
	$$
	So, computing the integral and the norm in the right-hand side, we find
	$$
	u^N_i(t,\bo{x}+\bo{v})=u^N_i(t,\bo{x})+\frac 1{N-1}\sum\limits_{j\neq i}D_m U(t,x_i,\meh,x_j)\cdot v_j + o(|v|)\,.
	$$
	This Taylor expansion proves the first formula for all points $\bo{x}$ with $x_j\neq x_k$ when $j\neq k$. Since this subset is dense in $\Omega^N$, we have proved \eqref{eq:reg_uN_2} and \eqref{eq:reg_uN_3}.
	
	As regards \eqref{eq:reg_uN_4}, we start showing that $D_{x_j}u^N_i$ is a Lipschitz function in the space variable. Actually
	\begin{equation*}
	\begin{split}
	\left|D_{x_j} u^N_i(t,\bo{x})-D_{x_j}u^N_i(t,\bo{y})\right|\le &\,\frac CN \left|D_m U(t,x_i,\meh,x_j)-D_mU(t,x_i,\meh,y_j)\right|\\
	+ &\,\frac CN\left|D_m U(t,x_i,\meh,y_j)-D_mU(t,x_i,m_{\bo{y}}^{N_i},y_j)\right|\,.
	\end{split}
	\end{equation*}
	The first term in the right-hand side is immediately controlled by $\frac CN|\bo{x}-\bo{y}|$, using the regularity of $\dm{U}$ \eqref{eq:regdmU}. As regards the second term, we use \eqref{eq:LipdmU} to obtain
	$$
	\frac CN\left|D_m U(t,x_i,\meh,y_j)-D_mU(t,x_i,m_{\bo{y}}^{N_i},y_j)\right|\le \frac CN\dw\big(\meh,m_{\bo{y}}^{N_i}\big)\le \frac CN|\bo{x}-\bo{y}|\,.
	$$
	This means that $D^2_{x_j,x_j}$ exists almost everywhere, and
	$$
	\norminf{D^2_{x_jx_j}}\le\frac CN\,.
	$$
	To prove \eqref{eq:reg_uN_4}, we estimate the quantity
	$$
	\left|\frac{D_{x_j}u^{N,i}(t,\bo{x}+h\bo{e_{jk}})-D_{x_j}u^N_i(t,\bo{x})}h-\frac 1{N-1}\partial_{y_k}D_m U(t,x_i,\meh,x_j)\right|\,,
	$$
	where $\bo{e_{jk}}=\left(e^1_{jk},\dots,e^N_{jk}\right)$, with $e^l_{jk}=0$ if $l\neq k$ and $e^j_{jk}=e_k\in\R^d\,.$
	We observe that we can do that as long as $x_{j}\notin\partial\Omega$, in order to have $\bo{x}+h\bo{e_{jk}}\in\Omega^N$ for $h$ small enough. Since this subset is dense in $\Omega^N$, we can restrict ourselves to this case.
	
	Using \eqref{eq:reg_uN_2}, we can bound the quantity above by
	\begin{align*}
	&\frac CN\left|\frac{D_mU(t,x_i,m_{\bo{x}+h\bo{e_{jk}}}^{N,i},x_j+he_k)-D_mU(t,x_i,\meh,x_j+he_k)}h\right|\\
	+ &\frac CN \left|\frac{D_mU(t,x_i,\meh,,x_j+he_k)-D_mU(t,x_i,\meh,x_j)}h-\partial_{y_k}D_m U(t,x_i,\meh,x_j)\right|\,.
	\end{align*}

	The first term is estimated from above, using \eqref{eq:LipdmU}, by
	$$
	\frac C{Nh}\,\dw\!\left(m_{\bo{x}+h\bo{e_{jk}}}^{N,i},\meh\right)\le\frac C{N^2}\,,
	$$
	while the second term, using Lagrange's Theorem and \eqref{eq:LipdmU}, is equal to
	$$
	\frac CN\left|\partial_{y_k}D_m U(t,x_i,\meh,x^h)-\partial_{y_k}D_m U(t,x_i,\meh,x_j)\right|\le\frac CN h^\alpha\,,
	$$
	for a certain $x^h$ in the line segment between $x_j$ and $x_j+he_k$. Then, for $h$ sufficiently small, we obtain
	$$
	\left|\frac{D_{x_j}u^{N,i}(t,\bo{x}+h\bo{e_{jk}})-D_{x_j}u^N_i(t,\bo{x})}h-\frac 1{N-1}\partial_{y_k}D_m U(t,x_i,\meh,x_j)\right|\le \frac C{N^2}\,,
	$$
	and, passing to the limit as $h\to 0$ for all $1\le k\le d\,$, we obtain \eqref{eq:reg_uN_4} and we conclude.
	\end{proof}
	\end{prop}

	Now we are ready to state a first result, showing that $(u^N_i)_{1\le i\le N}$ is ``almost" a solution of the Nash system \eqref{nash}.
	
	\begin{thm}\label{thm:nash_uni}
	Let hypotheses \ref{ipotesi} be satisfied. Then $u^N_i\in\mathcal{C}^1([0,T]\times\Omega^N)$, $u^N_i(t,\cdot)\in W^{2,\infty}(\Omega^N)$ and $u^N_i$ solves almost everywhere the following equation:
	\begin{equation}\label{eq:nash_uni}
	\begin{cases}
		-\partial_t u^N_i(t,\bo x)-\mathlarger{\sum}\limits_j \mathrm{tr}(a(x_j)D^2_{x_j,x_j}u^N_i(t,\bo x)) +H(x_i,D_{x_i}u^N_i(t,\bo x))\\
		\hspace{2.2cm}+\,\,\mathlarger{\sum}\limits_{j\neq i}H_p(x_j,D_{x_j}u^N_j(t,\bo x))\cdot D_{x_j}u^N_i(t,\bo x)=F(\bo x,m_{\bo x}^{N,i})+r^N_i(t,\bo{x})\,,\\
		u^N_i(T,\bo x)=G(\bo x,m_{\bo x}^{N,i})\,,\\
		a(x_j)D_{x_j}u^{N,i}(t,\bo x)\cdot\nu(x_j)_{|x_j\in\partial\Omega}=0\,,
	\end{cases}
	\end{equation}
	where $r^N_i\in L^\infty$ with $\norminf{r^N_i}\le\frac CN\,.$
	\begin{proof}
	The regularity of $u^N_i$ follows from \eqref{eq:reg_uN_1} and Proposition \ref{prop:reg_uN_2}.
	
	The boundary condition of $u^N_i$ is an immediate consequence of the representation formula for the derivatives of $u^N_i$ and the boundary conditions of \eqref{ME}.
	
	Actually for $j=i$ we have
	$$
	a(x_i)D_{x_i}u^N_i\cdot\nu(x_i)_{|x_i\in\partial\Omega}=a(x_i)D_xU(t,x_i,\meh)\cdot\nu(x_i)_{|x_i\in\partial\Omega}=0
	$$
	thanks to the first boundary condition of \eqref{ME}. On the other hand, for $j\neq i$ we have
	$$
	a(x_j)D_{x_j}u^N_i\cdot\nu(x_j)_{|x_j\in\partial\Omega}=\frac 1{N-1}a(x_j)D_mU(t,x_i,\meh,x_j)\cdot\nu(x_j)_{|x_j\in\partial\Omega}=0
	$$
	thanks to the second boundary condition of \eqref{ME}.
	
	We sketch the rest of the proof, which is the same as \emph{Proposition 6.3} of \cite{cardSI}. Evaluating the Master Equation \eqref{ME} at $(t,x_i,\meh)$ and using Proposition \ref{prop:reg_uN_2}, we find
	
	\begin{align*}
	&-\partial_t u^N_i-\mathrm{tr}(a(x_i)D^2_{x_i,x_i}u^N_i)+H(x_i,D_{x_i}u^N_i) -\into\mathrm{tr}(a(y)D_yD_mU(t,x_i,\meh,y))d\meh(y)\\
	&+\frac 1{N-1}\mathlarger{\sum}\limits_{j\neq i}H_p(x_j,D_xU(t,x_j,\meh))\cdot D_mU(t,x_i,\meh,x_j)=F(t,x_i,\meh)\,.
	\end{align*}
	Using the derivative formulas of $u^N_i$ and the Lipschitz continuity of $D_xU$ with respect to $m$, we get
	\begin{align*}
	\frac 1{N-1}\sum\limits_{j\neq i}H_p(x_j,D_xU(t,x_j,\meh))&\cdot D_mU(t,x_i,\meh,x_j)\\
	&=\sum\limits_{j\neq i}H_p(x_j, D_{x_j}u^N_j)\cdot D_{x_j}u^N_i(t,\bo{x}) + O\left(\frac 1N\right)\,.
	\end{align*}
	For the integral term, we have
	$$
	\into\mathrm{tr}(a(y)D_yD_mU(t,x_i,\meh,y))d\meh(y)=\sum\limits_{j\neq i}D^2_{x_j,x_j}u^N_i(t,\bo{x})+O\left(\frac 1N\right)\,.
	$$
	Collecting all the estimations we obtain \eqref{eq:nash_uni}, which concludes the proof.
	\end{proof}
	\end{thm}

\section{The convergence result}
	
	Now we turn to the main convergence result. To do that, we consider the functions $(u^N_i)_i$ and the solutions $(v^N_i)_i$  of the system \eqref{nash}. We note that these solutions are \emph{symmetrical}, i.e. there exist two functions $V^N$ and $U^N:\Omega\times\Omega^{N-1}\to\R$ such that, for all $x\in\Omega$, the functions $V^N(x,\cdot)$ and $U^N(x,\cdot)$ are invariant under permutations and, $\forall\,i=1,\dots,N$,
	\begin{align*}
	v^N_i(t,\bo{x})=V^N(x_i,(x_1,\dots,x_{i-1},x_{i+1},\dots,x_N))\,,\\
	u^N_i(t,\bo{x})=U^N(x_i,(x_1,\dots,x_{i-1},x_{i+1},\dots,x_N))\,.\\
	\end{align*}
	We fix $t_0\in[0,T)$, $m_0\in\mathcal{P}(\Omega)$ and $\bo{Z}=(Z^i)_i$ a family of i.i.d. random variables of law $m_0$. We consider the process $\bo{Y}_t=(Y_t^i)_i$ solution of the following system:
	\begin{equation}\label{eq:dynY}
	\begin{cases}
	dY_t^i=-H_p(Y_t^i,D_{x_i}v^N_i(t,\bo{Y}_t))\, dt+\sqrt 2\sigma(Y_t^i)\,dB_t^i\,-dk_t^i\,,\\
	Y_{t_0}^i=Z_i\,,
	\end{cases}
	\end{equation}
	where $k_t^i$ is a \emph{reflected process along the co-normal}.
	
	The last theorem before the main result is the following.
	\begin{thm}\label{thm:penultimo}
		Assume hypotheses \ref{ipotesi} hold. Then, for any $1\le i\le N\,$, we have
		\begin{equation}\label{eq:stimafeedback}
			\E\left[\int_{t_0}^T|D_{x_i}v^N_i(t,\bo{Y}_t)-D_{x_i}u^N_i(t,\bo{Y}_t)|^2\,dt\right]\le\frac{C}{N^2}\,.
		\end{equation}
		Moreover, $\P-a.s.\,$,
		\begin{equation}\label{eq:stimainizio}
			\big|u^N_i(t_0,\bo{Z})-v^N_i(t,\bo{Z})\big|\le\frac CN\,.
		\end{equation}
	\end{thm}
	
	The proof is almost exactly the same of \emph{Theorem 6.2.1} of \cite{cardSI}, but here we need to use an extension of the Ito's formula, with functions $\phi$ not necessarily $\mathcal{C}^2$ in the space variable, and with a reflection term in the process. This generalization is stated in the following Lemma.
	
	\begin{lem}\label{lem:prob}
	Let $\phi:[0,T]\times\Omega\to\R$ a $W^{2,\infty}$ function with respect to $x$ and a $\mathcal{C}^1$ function with respect to $t$, such that
	$$
	\bdone{\phi}=0\,.
	$$
	Let $m_0\in\mathcal{P}(\Omega)$ and let $X_t$ be a process in the probability space $(\tilde{\Omega},(\mathcal{F}_t)_t,\P)$, with initial density $m_0$, satisfying
	$$
	dX_t=b(t,X_t)\,dt+\sigma(X_t)\,dB_t-dk_t^i\,,
	$$
	where $(B_t)_t$ is a Brownian motion, $b$ and $\sigma$ are bounded functions respectively in $L^\infty$ and $\mathcal{C}^{1+\alpha}$, with $\sigma$ a uniformly elliptic matrix, and $k_t^i$ is a reflected process along the co-normal.
	
	Then the following formula holds $\forall t$ and a.s. in $\omega\in\tilde{\Omega}\,:$
	\begin{equation*}
	\begin{split}
	\phi(t,X_t)=\phi(0,X_0) & + \int_0^t\left(\phi_t(s,X_s)+\miezz\mathrm{tr}(a(X_s)D^2\phi(s,X_s))+b(s,X_s)\cdot D\phi(s,X_s)\right)ds\\
	& + \int_0^t\sigma(X_s)D\phi(s,X_s)\,dB_s\,.
	\end{split}
	\end{equation*}
	\begin{proof}
	We consider $\phi^n$ a sequence of $\mathcal{C}^{1,2,N}$ functions, bounded uniformly in $n$ together with their derivatives, such that $\phi^n\to\phi$ pointwise together with its first order derivatives in space and time, and almost everywhere for the second order derivatives in space.
	
	We define $a=\sigma\sigma^*$. The classical Ito's formula for $\phi^n$ tells us that
	\begin{equation*}
	\begin{split}
		\phi^n(t,X_t)=\phi^n(0,X_0) & + \int_0^t\left(\phi^n_t(s,X_s)+\miezz\mathrm{tr}(a(X_s)D^2\phi^n(s,X_s))+b(s,X_s)\cdot D\phi^n(s,X_s)\right)ds\\
		& + \int_0^t\sigma(X_s)D\phi^n(s,X_s)\,dB_s-\int_0^t a(X_s)D\phi^n(s,X_s)\nu(X_s)\,d|k|_s\,,
		\end{split}
	\end{equation*}
	and so, since $\phi^n$ satisfies $a(x)D\phi^n\cdot\nu_{|\partial\Omega}=0\,,$
	\begin{equation*}
	\begin{split}
		\phi^n(t,X_t)=\phi^n(0,X_0) & + \int_0^t\left(\phi^n_t(s,X_s)+\miezz\mathrm{tr}(a(X_s)D^2\phi^n(s,X_s))+b(s,X_s)\cdot D\phi^n(s,X_s)\right)ds\\
		& + \int_0^t\sigma(X_s)D\phi^n(s,X_s)\,dB_s\,.
	\end{split}
	\end{equation*}
	Since $\phi^n\to\phi$ pointwise, we can pass to the limit for the terms outside the integrals.
	
	For the term in the deterministic integral, we note that the law $m(t)$ of the process $X_t$ satisfies the following Fokker-Planck equation:
	$$
	\begin{cases}
	m_t-\mathrm{div}(a(x)Dm)-\mathrm{div}(m\tilde{b})=0\,,\\
	m(0)=m_0\,,\\
	\left[a(x)Dm+m\tilde{b}\right]\cdot\nu_{|\partial\Omega}=0\,,
	\end{cases}
	$$
	with $\tilde{b}(x)=b(x)+\big(\sum\limits_j\partial_{x_j} a_{ji}(x)\big)_i$. So, thanks to \emph{Proposition 5.3} of \cite{memedesimoSI}, we have that $m$ is globally bounded in $L^p(Q_T)$ for some $p>1\,.$ Hence, we have
	\begin{align*}
	&\E\left[\int_0^t\left|\phi^n_t(s,X_s)+\miezz\mathrm{tr}(a(X_s)D^2\phi^n(s,X_s))+b(s,X_s)\cdot D\phi^n(s,X_s)\right.\right.\\
	&\qquad\left.\left.-\phi_t(s,X_s)-\miezz\mathrm{tr}(a(X_s)D^2\phi(s,X_s))-b(s,X_s)\cdot D\phi(s,X_s)\right|ds\right]\\
	&\qquad\le \inti\left(|\phi^n_t-\phi_t|+\miezz|\mathrm{tr}(aD^2(\phi^n-\phi))|+|b\cdot(D\phi^n-D\phi)|\right)m(s,x)\,dxds\to0\,,
	\end{align*}
	where the dominated convergence is guaranteed by the $a.e.$ convergence of $\phi^n_t,D\phi^n,D^2\phi^n$ and the global boundedness of $m$ in $L^p$ and of $\phi^n_t$, $D\phi^n$ and $D^2\phi^n$ in $L^\infty$.
	
	As regards the last term, the $a.s.$ convergence is guaranteed by the property of the stochastic integral. Actually, we have
	\begin{align*}
	\E\left[\left(\int_0^t\sigma_s D\phi^n(s,X_s)\,dB_s-\int_0^t\sigma_sD\phi(s,X_s)dB_s\right)^2\right] &=\E\left[\int_0^t|\sigma_s(D\phi^n-D\phi)(s,X_s)|^2\,ds\right]\\
	&\le C\norminf{D\phi^n-D\phi}^2\to0\,.
	\end{align*}
	This concludes the Lemma.
	\end{proof}
	\end{lem}
	
	We note that, if there are no reflection, the boundary condition for $\phi$ can be removed.

	\begin{proof}[Sketch of the Proof of Theorem \ref{thm:penultimo}]

	Without loss of generality, we work with $t_0=0$ and we start proving \eqref{eq:stimafeedback}. To simplify the rest of the proof, we will use the following notations:
	$$
	\begin{array}{ll}
	\ui=u^N_i(t,\bo{Y}_t)\,, &\qquad \duij=D_{x_j}u^N_i(t,\bo{Y}_t)\,,\\
	\vi=v^N_i(t,\bo{Y}_t)\,, &\qquad \dvij=D_{x_j}v^N_i(t,\bo{Y}_t)\,.
	\end{array}
	$$
	
	Since $u^N_i$ and $v^N_i$ satisfy Neumann boundary conditions, we can use Lemma \ref{lem:prob} on the processes $\ui$ and $\vi$. Then, we apply the classical Ito's formula to the process $\big(\ui-\vi\big)^2$ to obtain
	$$
	d\left(\ui-\vi\right)^2=(A_t+B_t)\,dt+2\sqrt 2\left(\ui-\vi\right)\sum\limits_j\left[\sigma(Y_t^i)(\dvij-\duij)\right]\,dB_t^j\,,
	$$
	where
	\begin{align*}
	A_t &= 2(\ui-\vi)\left(H(Y_t^i,\duii)-H(Y_t^i,\dvii)\right)\\
	&- 2(\ui-\vi)\left(\duii\big(H_p(Y_t^i,\duii)-H_p(Y_t^i,\dvii)\big)\right)\\
	&- 2(\ui-\vi)\left((\duii-\dvii)H_p(Y_t^i,\dvii)-r^N_i(t,\bo{Y}_t)\right)
	\end{align*}
	and
	\begin{align*}
	B_t &= 2\sum\limits_j a(Y_t^i)(\duij-\dvij)\cdot(\duij-\dvij)\\
	&- 2(\ui-\vi)\sum\limits_j \duij\left(H_p(Y_t^j,\dvjj)-H_p(Y_t^j,\dujj)\right)\,.
	\end{align*}

	Now we integrate from $t$ to $T$ the above formula and take the conditional expectation given $\bo{Z}$. Using hypothesis $a.$ of \ref{ipotesi} and the previous results, in particular the bounds on $\duij$ and $r^N_i$, we get
	\begin{align*}
	&\,\E^{\bo{Z}}\left[|\ui-\vi|^2\right]+2\nu\sum\limits_j \E^{\bo{Z}}\left[\int_t^T|\duijs-\dvijs|^2\,ds\right]\\
	\le &\,\frac CN\int_t^T\E^{\bo{Z}}[|\uis-\vis|]\,ds+C\int_t^T\E^{\bo{Z}}[|\uis-\vis|\cdot|\duiis-\dviis|]\,ds\\
	+ &\,\frac CN\sum\limits_{j\neq i}\int_t^T\E^{\bo{Z}}[|\uis-\vis|\cdot|\dujjs-\dvjjs|]\,ds\,.
	\end{align*}
	By a standard convexity argument and a generalized Young's inequality,
	\begin{align}
	\E^{\bo{Z}}\left[|\ui-\vi|^2\right] & +\nu\E^{\bo{Z}}\left[\int_t^T|\duiis-\dviis|^2\,ds\right]\label{qui}\\
	\le \frac C{N^2}+C\int_t^T\E^{\bo{Z}}[|\uis-\vis|^2]\,ds & +\frac\nu{2N}\sum\limits_j\E^{\bo{Z}}\left[\int_t^T|\dujjs-\dvjjs|^2\,ds\right]\,.\notag
	\end{align}
	The last term in the right-hand side can be removed by taking the mean of the inequalities over $i\in 1,\dots,N$. Hence, taking the mean and using Gronwall's Lemma, we obtain
	$$
	\supo\left[\sum\limits_i \E^{\bo{Z}}\big[|\ui-\vi|^2\big]\right]+\sum\limits_i\E^{\bo Z}\left[\int_0^T|\duii-\dvii|^2\,dt\right]\le\frac CN\,.
	$$
	Using this estimation in \eqref{qui} and applying again Gronwall's Lemma, we get
	$$
	\supo \E^{\bo{Z}}\big[|\ui-\vi|^2\big]+\E^{\bo Z}\left[\int_t^T|\duiis-\dviis|^2\,ds\right]\le\frac C{N^2}\,.
	$$
	Taking the integral at $t=0$, we prove \eqref{eq:stimafeedback}. On the other hand, evaluating the term in the $sup$ at $t=0$, we prove \eqref{eq:stimainizio} and conclude the Theorem.
	\end{proof}

	Now we are ready to prove the main theorem of this chaper.
	
	\begin{proof}[Proof of Theorem \ref{thm:Teorema}]
	We start choosing $m_0=1$. Then, \eqref{eq:stimainizio} implies
	$$
	|U(t_0,Z_i,m_{\bo Z}^{N,i})-v^N_i(t_0,\bo Z)|\le \frac CN\qquad\P-a.s.\,.
	$$
	Since the support of $m_0$ is $\Omega$, this means, thanks to the continuity of $v^N_i$ and $U$, that
	$$
	|U(t_0,x_i,\meh)-v^N_i(t_0,\bo x)|\le\frac CN\qquad\forall\bo x\in\Omega^N\,.
	$$
	By te Lipschitz continuity of $U$, we have
	$$
	|U(t_0,x_i,\meh)-U(t_0,x_i,m_{\bo x}^N)|\le C\dw(\meh,m_{\bo x}^N)\le\frac CN\,.
	$$
	Putting together the last two inequalities, we obtain \eqref{eq:risult1}.
	
	To prove \eqref{eq:risult2}, we use the results of \cite{2SI,30SI,35SI} in order to obtain
	$$
	\int_{\Omega^{N-1}}|u^N_i(t_0,\bo x)-U(t_0,x_i,m_0)|\prod\limits_{j\neq i}m_0(dx_j)\le C\int_{\Omega^{N-1}}\dw(\meh,m_0)\prod\limits_{j\neq i}m_0(dx_j)\le C\omega_N\,,
	$$
	where $\omega_N$ is defined in \eqref{eq:defomega}. With this inequality, we can conclude:
	\begin{align*}
	&\norm{w^N_i(t_0,\cdot,m_0)-U(t_0,\cdot,m_0)}_{L^1(m_0)}\\
	=&\into\left|\int_{\Omega^{N-1}}\big(v^N_i(t_0,\bo x)-U(t_0,x_i,m_0)\big)\prod\limits_{j\neq i}m_0(dx_j)\right|m_0(dx_i)\\
	\le &\E\left[|v^N_i(t,\bo Z)-u^N_i(t,\bo Z)|\right]+\int_{\Omega^N} |u^N_i(t_0,\bo x)-U(t,x_i,m_0)|\prod\limits_{j\neq i} m_0(dx_j)\\
	\le &\frac CN+C\omega_N\le C\omega_N\,.
	\end{align*}
	\end{proof}
	
	\section{Convergence of the trajectories}
	We conclude with a last result concerning the convergence of the trajectories. To do that, we fix as before $t_0\in[0,T)$, $m_0\in\mathcal P(\Omega)$ and $\bo Z=(Z^i)_i$ a family of i.i.d. random variables of law $m_0$.
	
	We consider the process $\bo X_t=(X_t^i)_i$, solution of the following system:
	\begin{equation}\label{eq:defX}
	\begin{cases}
	dX_t^i=-H_p(X_t^i,D_{x_i}u^N_i(t,\bo X_t))\,dt+\sqrt 2\sigma(X_t^i)dB_t^i-a(X_t^i)\nu(X_t^i)d|k|_t^{i,X}\,,\\
	X_{t_0}^i=Z^i\,,
	\end{cases}
	\end{equation}
	and the process $\bo Y_t=(Y_t^i)_i$, already defined as the solution of
	\begin{equation*}
		\begin{cases}
			dY_t^i=-H_p(Y_t^i,D_{x_i}v^N_i(t,\bo Y_t))\,dt+\sqrt 2\sigma(Y_t^i)dB_t^i-a(Y_t^i)\nu(Y_t^i)d|k|_t^{i,Y}\,,\\
			Y_{t_0}^i=Z^i\,.
		\end{cases}
	\end{equation*}
	Here, $k_t^{i,X}$ and $k_t^{i,Y}$ denote respectively the reflected process along the co-normal for the processes $(X_t^i)_t$ and $(Y_t^i)_t\,.$
	
	The last theorem we want to prove is the following:
	\begin{thm}
	Assume hypotheses \ref{ipotesi} hold. Assume, moreover, that $a=\sigma\sigma^*$ with $\sigma\in W^{1,\infty}(\overline\Omega)$ and $a\in\mathcal{C}^2(\overline\Omega)\,$. Then, for any $1\le i\le N$, we have
	\begin{equation}\label{eq:ultimathule}
	\supt\E\left[|X_t^i-Y_t^i|^2\right]\le \frac C{N^2}\,.
	\end{equation}
	\end{thm}
	A similar result was proved in the same article \cite{cardSI}. Unfortunately, the reflecting term in the processes does not allow to follow and readapt the same ideas.
	
	To prove this result, we will be inspired by the ideas developed in \cite{cardconvtrajSI}.
	
	\begin{proof}
	Let $d(\cdot)$ be the oriented distance function from the boundary $\partial\Omega$, defined in this way:
	$$
	d(x)=\left\{
	\begin{array}{rl}
	 \mathrm{dist}(x,\partial\Omega) & \quad x      \in\partial\Omega\,,\\
	-\mathrm{dist}(x,\partial\Omega) & \quad x\notin\partial\Omega\,.
	\end{array}
	\right.
	$$
	
	Thanks to \cite{cingulSI} and the regularity of $\Omega$ we know that $d\in\mathcal{C}^{2+\alpha}$ in a neighbourhood of the boundary of $\Omega$. Hence, a classical regularizing argument allows us to consider a non-negative $\mathcal{C}^{2+\alpha}(\overline\Omega)$, called again $d$, which coincides with the oriented distance in a neighbourhood of the boundary.
	
	We want to apply Ito's formula to the following quantity:
	$$
	\psi(t,X_t^i,Y_t^i):=e^{-\alpha(\delta t+d(X_t^i)+d(Y_t^i))}\, \langle (a^{-1}(X_t^i)+a^{-1}(Y_t^i))(X_t^i-Y_t^i),X_t^i-Y_t^i\rangle\,, 
	$$
	for large $\alpha,\delta>0$ which will e chosen later. We start computing the first and second derivatives of $\psi.$ To do that, we introduce the following notations:
	\begin{itemize}
		\item For a differentiable matrix-valued function $a(\cdot):\Omega\to\R^{d\times d}$, $x\in\Omega$ and $v\in\R^d$, the quantity $Da(x)(v)$ stands for the matrix whose lines are the derivatives $(\partial_{x_k}a(x)v)_k$, whereas for $w\in\R^d$ the quantity $\langle Da(x)(v),w\rangle$ will denote the vector $(\langle \partial_{x_k}a(x)v,w\rangle)_k\,;$
		\item For $a(\cdot)$ as before, $x\in\Omega$ and $v,w\in\R^d$, the quantity $\langle D^2a(x)(v),w\rangle$ stands for the matrix whose entries are given by $(\langle \partial^2{x_ix_j}a(x)v,w\rangle)_{i,j}\,;$
		\item For $a(\cdot)$ as before, $x\in\Omega$ and $v,w\in\R^d$, the quantity $\langle D^2 a(x)(v),w\rangle$ stands for the matrix whose entries are given by $(\langle\partial^2{x_ix_j}a(x)v,w\rangle)_{i,j}\,;$
		\item For $v,w\in\R^d$, the tensor product $v\otimes w$ will denote the matrix $v^T w=(v_iw_j)_{i,j}\,;$
		\item For a matrix $A\in\R^{d\times d}$ and $1\le k\le d$, the vector corresponding to the $k$-th lines of $A$ will be denoted by $A_k$.
	\end{itemize}

	With these notations, we can start with the computation of the first order derivatives:
	\begin{align*}
	\partial_t\psi(t,x,y) &= -\alpha\delta\psi(t,x,y) = -\alpha\delta e^{-\alpha(\delta t+d(x)+d(y))}\langle \left(a^{-1}(x)+a^{-1}(y)\right)(x-y),x-y\rangle\,;\\
	D_x\psi(t,x,y) &=\espot\left[-\alpha\langle\left(a^{-1}(x)+a^{-1}(y)\right)(x-y),x-y\rangle Dd(x)\right.\\
	&\left.+\,\langle Da^{-1}(x)(x-y),x-y\rangle+2(a^{-1}(x)+a^{-1}(y))(x-y)\right]\,;\\
	D_y\psi(t,x,y) &=\espot\left[-\alpha\langle\left(a^{-1}(x)+a^{-1}(y)\right)(x-y),x-y\rangle Dd(y)\right.\\
	&\left.+\,\langle Da^{-1}(y)(x-y),x-y\rangle-2(a^{-1}(x)+a^{-1}(y))(x-y)\right]\,;
	\end{align*}
	As regards the second order derivatives, we have
	\begin{align*}
	D^2_{xx}\psi(t,x,y) &=\espot \left[\alpha^2\langle\left(a^{-1}(x)+a^{-1}(y)\right)(x-y),x-y\rangle Dd(x)\otimes Dd(x)\right.\\
	&- \alpha\langle Da^{-1}(x)(x-y),x-y\rangle\otimes Dd(x)-2\alpha\left(a^{-1}(x)+a^{-1}(y)\right)(x-y)\otimes Dd(x)\\
	&- \alpha\langle\left(a^{-1}(x)+a^{-1}(y)\right)(x-y),x-y\rangle D^2d(x)+\langle D^2 a^{-1}(x)(x-y),x-y\rangle\\
	&- \alpha Dd(x)\otimes\langle Da^{-1}(x)(x-y),x-y\rangle -2\alpha Dd(x)\otimes\left[(a^{-1}(x)+a^{-1}(y))(x-y)\right]\\
	&+ \left.4Da^{-1}(x)(x-y)+2a^{-1}(x)+2a^{-1}(y)\right]\,,\\\\
	D^2_{xy}\psi(t,x,y) &=\espot \left[\alpha^2\langle\left(a^{-1}(x)+a^{-1}(y)\right)(x-y),x-y\rangle Dd(x)\otimes Dd(y)\right.\\
	&- \alpha\langle Da^{-1}(x)(x-y),x-y\rangle\otimes Dd(y)-2\alpha\left(a^{-1}(x)+a^{-1}(y)\right)(x-y)\otimes Dd(y)\\
	&- \alpha Dd(x)\otimes\langle Da^{-1}(y)(x-y),x-y\rangle +2\alpha Dd(x)\otimes\left[(a^{-1}(x)+a^{-1}(y))(x-y)\right]\\
	&- \left.2Da^{-1}(x)(x-y)+2Da^{-1}(y)(x-y)-2a^{-1}(x)-2a^{-1}(y)\right]\,,\\\\
	D^2_{yy}\psi(t,x,y) &=\espot \left[\alpha^2\langle\left(a^{-1}(x)+a^{-1}(y)\right)(x-y),x-y\rangle Dd(y)\otimes Dd(y)\right.\\
	&- \alpha\langle Da^{-1}(y)(x-y),x-y\rangle\otimes Dd(y)+2\alpha\left(a^{-1}(x)+a^{-1}(y)\right)(x-y)\otimes Dd(y)\\
	&- \alpha\langle\left(a^{-1}(x)+a^{-1}(y)\right)(x-y),x-y\rangle D^2d(y)+\langle D^2 a^{-1}(y)(x-y),x-y\rangle\\
	&- \alpha Dd(y)\otimes\langle Da^{-1}(y)(x-y),x-y\rangle +2\alpha Dd(y)\otimes\left[(a^{-1}(x)+a^{-1}(y))(x-y)\right]\\
	&- \left.4Da^{-1}(y)(x-y)+2a^{-1}(x)+2a^{-1}(y)\right]\,.
	\end{align*}

	Now we are ready to use Ito's formula. Since $X_{t_0}=Y_{t_0}$, and taking into account that, $\forall\, A\in\mathrm{Sym}(\R^{d\times d})$, $\forall\, v,w\in\R^d$,
	$$
	\mathrm{tr}(A(v\otimes w))=\mathrm{tr}((Av)\otimes w)=\langle Av,w\rangle\,,
	$$
	we obtain
	\begin{align}
	\E[\psi(t,X_t^i,Y_t^i)] &= \E\left[\int_{t_0}^t\espos\langle\big(a^{-1}(X_s)+a^{-1}(Y_s)\big)(X_s-Y_s),X_s-Y_s\rangle A_s\,ds\right]\notag\\
	&+\E\left[\int_{t_0}^t\espos\big(B_s+C_s+D_s+E_s+F_s+G_s+H_s\big)\,ds\right]\label{eq:ito}\\
	&+\E\left[\int_{t_0}^t\espos\big(L_s\,d|k|_s^{i,X}+M_s\,d|k|_s^{i,Y}\big)\right]\notag\,,
	\end{align}
	where
	\begin{align*}
	A_s =&-\alpha\delta+\alpha H_p(X_s^i,D_{x_i}u^N_i(s,\bo X_s))\cdot Dd(X_s^i) + \alpha H_p(Y_s^i,D_{x_i}v^N_i(s,\bo Y_s))\cdot Dd(Y_s^i)\\
	&+\alpha^2|\sigma^*(X_s^i)Dd(X_s^i)+\sigma^*(Y_s^i)Dd(Y_s^i)|^2-\alpha\,\mathrm{tr}(a(X_s^i)D^2d(X_s^i))-\alpha\,\mathrm{tr}(a(Y_s^i)D^2d(Y_s^i))\,;\\
	B_s =&-\left\langle Da^{-1}(X_s^i)(X_s^i-Y_s^i),X_s^i-Y_s^i\right\rangle\cdot H_p(X_s^i,D_{x_i}u^N_i(t,\bo{X_s}))\\
	&-\left\langle Da^{-1}(Y_s^i)(X_s^i-Y_s^i),X_s^i-Y_s^i\right\rangle\cdot H_p(Y_s^i,D_{x_i}v^N_i(t,\bo{Y_s}))\,;\\
	C_s =&-2\alpha\left\langle \sigma^*(X_s^i)Dd(X_s^i)+\sigma^*(Y_s^i)Dd(Y_s^i),\sigma^*(X_s^i)\langle Da^{-1}(X_s^i)(X_s^i-Y_s^i),X_s^i-Y_s^i\rangle\right\rangle\\
	&-2\alpha\left\langle \sigma^*(X_s^i)Dd(X_s^i)+\sigma^*(Y_s^i)Dd(Y_s^i),\sigma^*(Y_s^i)\langle Da^{-1}(Y_s^i)(X_s^i-Y_s^i),X_s^i-Y_s^i\rangle\right\rangle;\\
	D_s =&\hspace{0.5cm} 2(a^{-1}(X_s^i)+a^{-1}(Y_s^i))(X_s^i-Y_s^i)\cdot\left(H_p(Y_s^i,D_{x_i}v^N_i(s,\bo{Y}_s))-H_p(X_s^i,D_{x_i}u^N_i(s,\bo{X}_s))\right)\,;\\
	E_s=&-4\alpha\left\langle \sigma^*(X_s^i)Dd(X_s^i)+\sigma^*(Y_s^i)Dd(Y_s^i),(\sigma^*(X_s^i)-\sigma^*(Y_s^i))(a^{-1}(X_s^i)+a^{-1}(Y_s^i))(X_s^i-Y_s^i)\right\rangle;\\
	F_s=&\hspace{0.5cm}\mathrm{tr}(a(X_s^i)\langle D^2a^{-1}(X_s^i)(X_s^i-Y_s^i),X_s^i-Y_s^i\rangle)+\mathrm{tr}(a(Y_s^i)\langle D^2a^{-1}(Y_s^i)(X_s^i-Y_s^i),X_s^i-Y_s^i\rangle);\\
	G_s=&\hspace{0.5cm}4\left\langle\left(\mathrm{tr}\left((\sigma(X_s^i)-\sigma(Y_s^i))\left(\sigma^*(X_s^i)\mathrm{Jac}(a_k^{-1}(X_s^i))+sigma^*(X_s^i)\mathrm{Jac}(a_k^{-1}(X_s^i))\right)_k\right)\right),X_s^i-Y_s^i\right\rangle;\\
	H_s=&\hspace{0.5cm}\mathrm{tr}\left((\sigma(X_s^i)-\sigma(Y_s^i))(\sigma^*(X_s^i)-\sigma^*(Y_s^i))(a^{-1}(X_s^i)+a^{-1}(Y_s^i))\right)\,;\\
	L_s=&\hspace{0.5cm}a(X_s^i)Dd(X_s^i)\cdot\left[-\alpha Dd(X_s^i)\left\langle(a^{-1}(X_s^i)+a^{-1}(Y_s^i))(X_s^i-Y_s^i),X_s^i-Y_s^i\right\rangle\right.\\
	&+\left.\left\langle Da^{-1}(X_s^i)(X_s^i-Y_s^i),X_s^i-Y_s^i\right\rangle+2(a^{-1}(X_s^i)+a^{-1}(Y_s^i))(X_s^i-Y_s^i)\right]\,;\\
	M_s=&\hspace{0.5cm}a(Y_s^i)Dd(Y_s^i)\cdot\left[-\alpha Dd(Y_s^i)\left\langle(a^{-1}(X_s^i)+a^{-1}(Y_s^i))(X_s^i-Y_s^i),X_s^i-Y_s^i\right\rangle\right.\\
	&+\left.\left\langle Da^{-1}(Y_s^i)(X_s^i-Y_s^i),X_s^i-Y_s^i\right\rangle+2(a^{-1}(X_s^i)+a^{-1}(Y_s^i))(X_s^i-Y_s^i)\right]\,.
	\end{align*}
	We take $\alpha>1$ and we start by analyzing the terms in the deterministic part. As regards $A_s$, using the boundedness of the coefficients we immediately obtain
	$$
	A_s\le-\alpha\delta+C\alpha^2\,.
	$$
	For the other terms, the hypotheses on the coefficients (in particular the Lipschitz bound on $\sigma$) easily implies
	$$
	|B_s|+|C_s|+|E_s|+|F_s|+|G_s|+|H_s|\le C|X_s^i-Y_s^i|^2\,.
	$$
	As regards $D_s$, we have
	\begin{align*}
	|D_s| &\le C|X_s^i-Y_s^i|\cdot\big|H_p(Y_s^i,D_{x_i}v^N_i(s,\bo Y_s))-H_p(Y_s^i,D_{x_i}u^N_i(s,\bo Y_s))\big|\\
	&+C|X_s^i-Y_s^i|\cdot\big|H_p(Y_s^i,D_{x_i}u^N_i(s,\bo Y_s))-H_p(X_s^i,D_{x_i}u^N_i(s,\bo Y_s))\big|\\
	&+C|X_s^i-Y_s^i|\cdot\big|H_p(X_s^i,D_{x_i}u^N_i(s,\bo Y_s))-H_p(X_s^i,D_{x_i}u^N_i(s,\bo X_s))\big|\\
	&\le C|X_s^i-Y_s^i|\cdot\left(|X_s^i-Y_s^i|+\frac 1N\sum\limits_{j\neq i}|X_s^j-Y_s^j|+|D_{x_i}v^N_i(s,\bo Y_s)-D_{x_i}u^N_i(s,\bo Y_s)|\right)\\
	&\le C|X_s^i-Y_s^i|^2+\frac CN\sum\limits_{j\neq i}\big(|X_s^i-Y_s^i|\cdot|X_s^j-Y_s^j|\big)+C|D_{x_i}v^N_i(s,\bo Y_s)-D_{x_i}u^N_i(s,\bo Y_s)|^2\,.
	\end{align*}
	Now we focus ourselves on the reflecting terms. The uniform ellipticity of $a$ implies
	$$
	\langle a(X_s^i)Dd(X_s^i),Dd(X_s^i)\rangle\ge\lambda\,,\quad\langle(a^{-1}(X_s^i)+a^{-1}(Y_s^i))(X_s^i-Y_s^i),X_s^i-Y_s^i\rangle\ge2\mu^{-1}|X_s^i-Y_s^i|^2\,.
	$$
	We obtain
	\begin{align*}
	L_s\le -2\alpha\lambda\mu^{-1}|X_s^i-Y_s^i|^2+C|X_s^i-Y_s^i|^2 &+2Dd(X_s^i)\cdot(X_s^i-Y_s^i)\\
	&+2Dd(X_s^i)\cdot\big(a(X_s^i)a^{-1}(Y_s^i)(X_s^i-Y_s^i)\big)\,.
	\end{align*}
	The last term can be written as
	$$
	2Dd(X_s^i)\cdot(X_s^i-Y_s^i)+2Dd(X_s^i)\cdot[(a(X_s^i)-a(Y_s^i))a^{-1}(Y_s^i)(X_s^i-Y_s^i)]\,.
	$$
	This estimate implies, up to changing the constant $C$,
	$$
	L_s\le(-2\alpha\lambda\mu^{-1}+C)|X_s^i-Y_s^i|^2+4Dd(X_s^i)\cdot(X_s^i-Y_s^i)\,.
	$$
	In order to estimate the right-hand side term, we use a Taylor expansion of the function $d(\cdot)$. We have
	$$
	d(X_s^i)=d(Y_s^i)+Dd(X_s^i)\cdot(X_s^i-Y_s^i)+\miezz D^2d(\xi)(X_s^i-Y_s^i)\cdot(X_s^i-Y_s^i)\,,
	$$
	where $\xi=X_s+k(Y_s-X_s)$, for a certain $0\le k\le 1$. So we get
	$$
	Dd(X_s^i)\cdot(X_s^i-Y_s^i)\le d(X_s^i)-d(Y_s^i)+C|X_s^i-Y_s^i|^2\,.
	$$
	Since the reflecting process $k^{i,X}_\cdot$ takes values in the set $\{X_\cdot\in\partial\Omega\}$, we have in this set $d(X_s^i)=0$. This means, up to changing $C$,
	$$
	L_s\le (-2\alpha\lambda\mu^{-1}+C)|X_s^i-Y_s^i|^2-4d(Y_s^i)\le 0\,,
	$$
	for $\alpha$ sufficiently large. In the same way we can prove $M_s\le 0$ in the set $\{Y_\cdot\in\partial\Omega\}\,$.
	
	Now we come back to \eqref{eq:ito}. The uniform ellipticity of $a$ implies, for a certain $C>0\,,$
	\begin{gather*}
	\psi(t,X_t^i,Y_t^i)\ge 2e^{-\alpha(\delta T+C)}\mu^{-1}|X_t^i-Y_t^i|^2\,,\\
	\left\langle\big(a^{-1}(X_s^i)+a^{-1}(Y_s^i)\big)(X_s^i-Y_s^i),X_s^i-Y_s^i\right\rangle\le2\lambda^{-1}|X_s^i-Y_s^i|^2\,.
	\end{gather*}
	Since $A_s\le 0$ for $\delta$ sufficiently large (depending on $\alpha$), collecting all the estimates we obtain, up to changing $C$,
	\begin{gather}
	2e^{-\alpha(\delta T+C)}\mu^{-1}\E[|X_t^i-Y_t^i|^2]\le\left(2\lambda^{-1}(-\alpha\delta+C\alpha^2)+C\right)\int_{t_0}^t\E[|X_s^i-Y_s^i|^2]\,ds\label{eq:itofatto}\\
	+\frac CN\int_{t_0}^t\E\left[\sum\limits_{j\neq i}(|X_s^i-Y_s^i|\cdot|X_s^j-Y_s^j|)\right]ds+C\int_{t_0}^t\E[|D_{x_i}v^N_i(s,\bo Y_s)-D_{x_i}u^N_i(s,\bo Y_s)|^2]ds\,.\notag
	\end{gather}
	The last term is immediately estimated with \eqref{eq:stimafeedback}. Moreover, since $\alpha>1$, the coefficient of the first term in the right-hand side can be written, up to changing $C$ depending on $\lambda$, as $2\lambda^{-1}(-\alpha\delta+C\alpha^2)\,$.
	
	Fix $K>0$, whose value will be given later. We choose $\delta>C\alpha+\frac{K\lambda}{2\alpha}$ and we have $2\lambda^{-1}(-\alpha\delta+C\alpha^2)\le -K\,$. Putting these estimations in \eqref{eq:itofatto} we obtain
	\begin{align}
	2e^{-\alpha(\delta T+C)}\mu^{-1}\E[|X_t^i-Y_t^i|^2]\le &-K\int_{t_0}^t\E[|X_s^i-Y_s^i|^2]\,ds\notag\\
	&+\frac CN\int_{t_0}^t\E\left[\sum\limits_{j\neq i}\big(|X_s^i-Y_s^i|\cdot|X_s^j-Y_s^j|\big)\right] ds+\frac C{N^2}\,.\label{eq:quasifinito}
	\end{align}
	Summing over $i\in[1,N]$, we get
		\begin{align*}
		2e^{-\alpha(\delta T+C)}\mu^{-1}\sum_{i=1}^N\E[|X_t^i-Y_t^i|^2]\le &-K\int_{t_0}^t\sum_{i=1}^N\E[|X_s^i-Y_s^i|^2]\,ds\\
		&+\frac CN\int_{t_0}^t\sum_{i,j=1}^N\E\left[|X_s^i-Y_s^i|\cdot|X_s^j-Y_s^j|\right] ds+\frac C{N}\,.
	\end{align*}
	We estimate separately the function in the last integral. We have
	$$
	\sum\limits_{i,j=1}^N\E[|X_s^i-Y_s^i|\cdot|X_s^j-Y_s^j|]=\E\left[\left(\sum_{i=1}^N|X_s^i-Y_s^i|\right)^2\right]\le N\sum_{i=1}^N\E[|X_s^i-Y_s^i|^2]\,,
	$$
	which implies
	$$
	2e^{-\alpha(\delta T+C)\mu^{-1}}\sum_{i=1}^N\E[|X_t^i-Y_t^i|^2]\le(C-K)\int_{t_0}^t\sum_{i=1}^N\E[|X_s^i-Y_s^i|^2]\,ds+\frac CN\,.
	$$
	We take $K>C$ and $\delta$ depending on $K$. This means, up to changing $C$,
	$$
	\sum_{i=1}^N\E[|X_t^i-Y_t^i|^2]\le\frac CN\,.
	$$
	With this information we can estimate the sum in \eqref{eq:quasifinito}:
	$$
	\E\left[\sum\limits_{j\neq i}|X_s^i-Y_s^i|\cdot|X_s^j-Y_s^j|\right]\le CN\,\E[|X_s^i-Y_s^i|^2]+\frac CN\,.
	$$
	Plugging these estimations in \eqref{eq:quasifinito} we finally get
	$$
	\E[|X_t^i-Y_t^i|^2]\le C\int_{t_0}^t\E[|X_s^i-Y_s^i|^2]\,ds+\frac C{N^2}\,.
	$$
	Using Gronwall's inequality we obtain \eqref{eq:ultimathule} and we conclude.
	\end{proof}	
		
	\vspace{2cm}
	\textbf{Acknowledgements.} 
	I wish to sincerely thank P. Cardaliaguet and A. Porretta for the help and the support during the preparation of this article. I wish to thank also F. Delarue for the enlightening ideas he gave to me. 
	

\begin{thebibliography}{abc}
		
		\bibitem{golSI} Achdou, Y., Buera, F. J., Lasry, J.-M., Lions, P.-L., Moll, B. (2014).  {\it Partial differential equation models in macroeconomics.} Phil. Trans. R Soc. A 372(2028):20130397. DOI: 10. 1098/rsta.2013.0397.
		
		\bibitem{2SI} Ajtai, M., Komlos, J., Tusn\'{a}dy, G. (1984). {\it On optimal matchings.} Combinatorica, 4(4), 259-264.
		
		\bibitem{ags} Ambrosio,  L., Gigli, N., Savar\'e, G. (2008). \textit{Gradient flows in metric spaces and in the space of probability measures. Second edition.} Lectures in Mathematics ETH Z\"{u}rich. Birkh\"{a}user Verlag, Basel
		
		\bibitem{dybala} Bayraktar, E., Cecchin, A., Cohen, A., Delarue, F. (2019). {\it Finite state mean field games with wright-fisher common noise}. arXiv preprint arXiv:1912.06701.
		
		\bibitem{nuova1SI} Bayraktar, E., Cohen, A. (2018). {\it Analysis of a finite state many player game using its master equation}. SIAM Journal on Control and Optimization, 56(5), 3538-3568.
		
		\bibitem{seiSI} Bensoussan, A., Frehse, J. (2002). \emph{Smooth solutions of systems of quasilinear parabolic equations.} ESAIM: Control, Optimisation and Calculus of Variations, 8, 169-193.
		
		\bibitem{14} Bensoussan, A., Frehse, J., Yam. S.C.P. (2015). {\it The Master Equation in mean field theory}. J. Math. Pures et Appliqu\'ees, 103, 1441-1474.
		
		\bibitem{15} Bensoussan, A., Frehse, J., Yam, S.C.P. (2017). {\it On the interpretation of the Master Equation}. Stoc. Proc. App., 127, 2093-2137.
		
		\bibitem{dieciSI} Bensoussan, A., Lions, P.-L. (1982). \emph{Contr\^{o}le Impulsionnel et In\'equations Quasi-Variationnelles}, Dunod, Paris.
		
		\bibitem{gomezSI} Bertucci, C. (2020). {\it Monotone solutions for mean field games master equations: finite state space and optimal stopping.} arXiv preprint arXiv:2007.11854.
		
		\bibitem{cardconvtrajSI} Briand, P., Cardaliaguet, P., Chaudru De Raynal, P.-E., Hu, Y. (2020). {\it Forward and Backward Stochastic Differential Equations with Normal Constraints in Law}. Stochastic Processes and their Applications, 130, 7021-7097.
		
		\bibitem{bucchinSI} Buckdahn, R., Li, J., Peng, S., Rainer, C. (2017). {\it Mean-field stochastic differential equations and associated PDEs.} Ann. Probab., 45, 824-878.
		
		\bibitem{28SI} Chassagneux, J.F., Crisan, D., Delarue, F. (2014). {\it Classical solutions to the Master Equation for large population equilibria}. arXiv preprint arXiv:1411.3009.
		
		\bibitem{cardconvSI} Cardaliaguet, P. (2017).  \textit{The convergence problem in mean field games with local coupling}. Applied Mathematics \& Optimization, 76(1), 177-215.
		
		\bibitem{nuova14SI}
		{ Cardaliaguet, P., Cirant, M., Porretta, A.} (2018). {\em Remarks on nash
			equilibria in mean field game models with a major player}, arXiv preprint
		arXiv:1811.02811.
		
		\bibitem{cardSI} Cardaliaguet, P., Delarue, F., Lasry, J.-M., Lions, P.-L. (2019). \textit{The Master Equation and the Convergence Problem in Mean Field Games}. Annals of Mathematics Studies, Vol. 2.
		
		\bibitem{resultunoSI} Carmona, R., Delarue, F. (2013). \textit{Probabilist analysis of Mean Field Games}. SIAM Journal on Control and Optimization, 51(4), 2705-2734.
		
		\bibitem{24} Carmona, R., Delarue, F. (2014). {\it The Master Equation for large population equilibriums}. Stochastic Analysis and Applications 2014, Editors: D. Crisan, B. Hambly, T. Zariphopoulou. Springer.
		
		\bibitem{resultunemmiezzSI} Carmona, R., Delarue, F. (2017). \textit{Probabilistic theory of mean field games with applications}. Springer Verlag.
		
		\bibitem{sonocarmelaSI} Carmona, R., Delarue, F. (2018). \textit{The Master Field and the Master Equation. Probabilistic Theory of
		Mean Field Games with Applications II}. Springer, Cham, 239-321.
		
		\bibitem{loackerSI} Carmona, R., Delarue, F., Lacker, D. (2016). {\it Probabilistic analysis of mean field games with a common noise}. Ann. Probab, 44, 3740-3803.
		
		\bibitem{cicciocaputoSI} Cecchin, A., Delarue, F. (2020). {\it Selection by vanishing common noise for potential finite state mean field games}. arXiv preprint arXiv:2005.12153.
		
		\bibitem{nuova11SI} Cecchin, A., Pelino, G. (2019). {\it Convergence, fluctuations and large deviations for finite state mean field games via the master equation}. Stochastic Processes and their Applications, 129(11), 4510-4555.
		
		\bibitem{checchinoSI} Cecchin, A., Pra, P.D., Fischer, M., Pelino, G. (2019).  {\it On the convergence problem in mean field games: a
		two state model without uniqueness.} SIAM Journal on Control and Optimization, 57(4), 2443-2466.
		
		\bibitem{nuova4SI} Delarue, F., Lacker, D., Ramanan, K. (2018). {\it From the master equation to mean field game limit theory: Large deviations and concentration of measure}. arXiv preprint arXiv:1804.08550.
		
		\bibitem{ramadanSI} Delarue, F.,  Lacker, D., Ramanan, K. (2019). {\it From the master equation to mean field game limit theory: a central limit theorem}. Electron. J. Probab. 24, no. 51, 1-54.
		
		\bibitem{cingulSI} Delfour, M.C., Zolesio, J.-P. (1994). {\it Shape analysis via oriented distance function}.  J.  Funct. Anal.   123,  129-201.
		
		\bibitem{30SI} Dereich, S., Scheutzow, M., Schottstedt, R. (2013). {\it Constructive quantization: approximation by empirical measures.} Annales de l'IHP, Probabilit\'es et Statistiques, 49(4), 1183-1203.
		
		\bibitem{jedi} Djete, M. F. (2020). \emph{Mean field games of controls: on the convergence of Nash equilibria} . arXiv preprint arXiv:2006.12993.
		
		\bibitem{chicazzosieteSI} Doncel, J., Gast, N., Gaujal, B. (2019).  {\it Discrete mean field games: Existence of equilibria and convergence.} Journal of Dynamics \& Games 6(3), 221-239.
		
		\bibitem{elshaaSI} El Karoui, N., and Chaleyat-Maurel, M. (1978). \emph{Un probl\`eme de r\'eflexion et ses applications au temps local et aux  \'equations diff\'erentielles stochastiques sur $\R$, cas continu}. Temps Locaux, Ast\'erisque, 52-53, 117-144.
		
		\bibitem{fishSI} Fischer, M. (2017). \emph{On the connection between symmetric n-player games and mean field games}. The
		Annals of Applied Probability, 27(2), 757-810.
		
		\bibitem{35SI} Fournier, N., Guillin, A. (2015). {\it On the rate of convergence in Wasserstein distance of the empirical measure}. Probability Theory and Related Fields, 162(3), 707-738.
		
		\bibitem{fifa21SI} Gangbo, W., M\'esz\'{a}ros, A. R. (2020). \textit{Global well-posedness of Master Equations for deterministic displacement convex potential mean field games.} arXiv preprint arXiv:2004.01660.
		
		\bibitem{tonaliSI} Gangbo, W., M\'{e}sz\'{a}ros, A. R., Mou, C., Zhang, J. (2021). \textit{Mean Field Games Master Equations with Non-separable Hamiltonians and Displacement Monotonicity.} arXiv preprint arXiv:2101.12362.
		
		\bibitem{nuova16SI}
		{Gangbo, W., Swiech, A. (2015).}, {\em Existence of a solution to an equation
			arising from the theory of mean field games}, Journal of Differential
		Equations, 259, pp.~6573--6643.
		
		\bibitem{gueantSI} Gu\'{e}ant, O., Lasry, J.-M., Lions, P.-L. (2011). \emph{Mean field games and applications}. In Paris-Princeton Lectures on Mathematical Finance 2010. Lecture Notes in Mathematics, Springer Berlin ed., 2003, 205-266.
		
		\bibitem{HCMSI} Huang, M.,  Caines, P.E.,  Malham\'e,  R.P. (2006).   {\it Large population stochastic dynamic games: closed-loop McKean-Vlasov systems and the Nash certainty equivalence principle},  Comm. Inf. Syst. {\bf 6}, 221--251.
		
		\bibitem{resultdueSI} Huang, M.,  Caines, P.E.,  Malham\'e,  R.P. (2007).   {\it Large population Cost-Coupled LQG Problems With Nonuniform Agents: Individual-Mass Behavior and Decentralized $\eps$-Nash Equilibria}. IEEE Transactions on Automatic Control, 52(9), 1560-1571.
		
		\bibitem{resulttreSI} Kolokoltsov, V.N., Li, J., Yang, W. (2011). \textit{Mean Field Games and nonlinear Markov Processes}. Preprint arXiv:1112.3744.
		
		\bibitem{loackerSI} Lacker, D. (2016). \textit{A general characterization of the mean field limit for stochastic differential games.} Probability Theory and Related Fields, 165, 581-648.
		
		\bibitem{loacker2SI} Lacker, D., (2020). \textit{On the convergence of closed-loop Nash equilibria to the mean field game limit}. Ann. Appl. Probab. 30(4): 1693-1761.
		
		\bibitem{lsuSI} Lady\v{z}enskaja, O.A., Solonnikov, V.A., Ural'ceva, N.N. (1967). \textit{Linear and Quasi-linear Equations of Parabolic Type}. Translations of Mathematical Monographs, Vol. 23, American Mathematical Society, Providence R.I..
		
		\bibitem{LL1SI} Lasry, J.-M., Lions, P.-L. (2006). {\it Jeux \`a champ moyen. I. Le cas stationnaire.}
		C. R. Math. Acad. Sci. Paris  343, 619--625.
		
		\bibitem{LL2SI} Lasry, J.-M., Lions, P.-L. (2006). {\it Jeux \`a champ moyen. II. Horizon fini et contr$\hat{o}$le optimal.}
		C. R. Math. Acad. Sci. Paris  343, 679--684.
		
		\bibitem{LL-japanSI} Lasry, J.-M., Lions, P.-L. (2007).  {\it Mean field games.}  Jpn. J. Math.  2 , no. 1, 229--260.
		
		\bibitem{LL3SI} Lasry, J.-M., Lions, P.-L., Gu\`eant, O. (2011). {\it Application of Mean Field Games to Growth Theory.}
		In: Paris-Princeton lectures on mathematical finance; Lecture notes in Mathematics. Springer, Berlin.
		
		\bibitem{prontoprontoprontoSI} Lions, P.-L. {\it Cours au Coll\`ege de France}. www.college-de-france.fr\,.
		
		\bibitem{55SI} Lions, P.-L., Menaldi., J.L., and Sznitman, A.S. (1981). \emph{Construction de processus de diffusion r\'efl\'echis par
		p\'enalisation du domaine}. Comptes-Rendus Paris, 292, 559-562.
		
		\bibitem{snitzmanSI} Lions, P.-L., Snitzman, A.S. (1984). \emph{Stochastic Differential Equations with Reflecting Boundary Conditions}. Communications on Pure and Applied Mathematics, 27. 511-537.
		
		\bibitem{koulibalySI} Mayorga, S. (2020). \textit{Short time solution to the master equation of a first order mean field game.} Journal of Differential Equations, 268(10), 6251-6318.
		
		\bibitem{durrSI} Nutz, M., Zhang, Y. (2019). {\it A mean field competition}. Mathematics of Operations Research, 44(4), 1245-1263.
		
		\bibitem{memedesimoSI} Ricciardi, M. (2021). \textit{The Master Equation in a Bounded Domain with Neumann Conditions}. Communications in Partial Differential Equations, DOI: 10.1080/03605302.2021.2008965

		\bibitem{63SI} Skorokhod, A. V. (1961). \emph{Stochastic equations for diffusion processes in a bounded region. 1}. Teor. Veroyatnost. i Primenen, 6(3), 287-298.
		
		\bibitem{63bisSI} Skorokhod, A. V. (1962). \emph{Stochastic equations for diffusion processes in a bounded region. 2}. Teor. Veroyatnost. i Primenen., 7(1), 5-25.
		
		\bibitem{64SI} Stroock, D. W., and Varadhan, S. R. S. (1971). \emph{Diffusion Processes with boundary conditions}. Comm. Pure Appl. Math., 24, 147-225.
		
		\bibitem{65SI} Tanaka, H. (1979). \emph{Stochastic differential equations with reflecting boundary condition in convex regions}. Hiroshima Math. J., 9, 163-177.
		
		
	\end{thebibliography}
\end{document}